\documentclass[11pt, reqno]{amsart}

\usepackage{amsmath, amsfonts, amssymb, amsbsy, bigstrut, graphicx, enumerate,  upref}

\usepackage[numbers, sort&compress]{natbib} 

\usepackage[breaklinks]{hyperref}

\usepackage{pstricks}
\usepackage{comment}
\usepackage{upref}
\usepackage{tikz}
\newcommand{\vol}{\mathrm{Vol}}

\newcommand{\ep}{\epsilon}

\newcommand{\mc}{\mathcal{C}}

\newcommand{\ran}{\operatorname{Ran}}

\newtheorem{thm}{Theorem}[section]
\newtheorem{lmm}[thm]{Lemma}
\newtheorem{cor}[thm]{Corollary}

\theoremstyle{remark}
\newtheorem{remark}[thm]{Remark}

\theoremstyle{definition}
\newtheorem{defn}[thm]{Definition}

\newcommand{\ee}{\mathbb{E}}

\newcommand{\mf}{\mathcal{F}}

\newcommand{\pp}{\mathbb{P}}

\newcommand{\ra}{\rightarrow}
\newcommand{\rr}{\mathbb{R}}

\newcommand{\tr}{\operatorname{Tr}}

\newcommand{\om}{\Omega}
\newcommand{\dbar}{{\overline{\om}}}
\numberwithin{equation}{section}
\newcommand{\la}{\left\langle}
\newcommand{\ral}{\right\rangle}
\newcommand{\re}{\mathbb{R}}
\renewcommand{\mc}[1]{\mathcal{#1}}
\newcommand{\Cc}{C_c^\infty}
\newcommand{\mfd}{\mf_{\Delta}}

\newcommand{\e}{\epsilon}
{
      \theoremstyle{plain}
      
 }

\usepackage[breaklinks]{hyperref}

\renewcommand{\O}[1]{{\textup{O}}_{#1}}
\renewcommand{\o}[1]{{\textup{o}}_{#1}}
\DeclareMathOperator*{\supp}{supp}
\begin{document}
\title[Eigenfunctions of perturbed Laplacians]{Arbitrarily small perturbations of Dirichlet Laplacians are quantum unique ergodic}
\author{Sourav Chatterjee}
\address{\newline Department of Statistics \newline Stanford University\newline Sequoia Hall, 390 Serra Mall \newline Stanford, CA 94305\newline \newline \textup{\tt souravc@stanford.edu}}
\author{Jeffrey Galkowski}
\address{\newline Department of Mathematics \newline Stanford University\newline 380 Serra Mall \newline Stanford, CA 94305\newline \newline \textup{\tt jeffrey.galkowski@stanford.edu}}
\thanks{Sourav Chatterjee's research was partially supported by NSF grant DMS-1441513}
\thanks{Jeffrey Galkowski's research was partially supported by the Mathematical Sciences
Postdoctoral Research Fellowship DMS-1502661.}

\begin{abstract}
Given an Euclidean domain with very mild regularity properties, we prove that there exist perturbations of the Dirichlet Laplacian of the form $-(I+S_\e)\Delta$ with $\|S_\e\|_{L^2\to L^2}\leq \e$ whose high energy eigenfunctions are quantum uniquely ergodic (QUE). Moreover, if we impose stronger regularity on the domain, the same result holds with $\|S_\e\|_{L^2\to H^\gamma}\leq \e$ for $\gamma>0$ depending on the domain. We also give a proof of a local Weyl law for domains with rough boundaries.
\end{abstract}

\maketitle

%\tableofcontents

\section{Introduction}\label{intro0}
In quantum mechanics, the Laplace operator on a manifold describes the behavior of a free quantum mechanical particle confined to the manifold. The eigenvalues of the Laplacian (under suitable boundary conditions) are the possible values of the energy of the particle and the eigenfunctions are the energy eigenstates. The square of an energy eigenstate gives the probability density function for the location of a particle with the given energy.

The subject of quantum chaos connects the properties of high energy eigenstates with the chaotic properties of the geodesic flow. One important result is the quantum ergodicity theorem due to  \citet{snirelman74}, \citet{colindeverdiere85}, and \citet{zelditch87} on manifolds without boundary and generalized to manifolds with boundary by \citet{gerardleichtnam93} and  \citet{zelditchzworski96}. The theorem states that if the geodesic flow on a manifold is ergodic, then almost all high energy eigenfunctions (in any orthonormal basis of eigenfunctions) equidistribute over the manifold in the sense that $|u|^2\to 1$ as a distribution. This phenomenon, or more precisely, its analog for equidistribution in both position and momentum, is known as  \emph{quantum ergodicity}.

The question of whether all (rather than almost all) high energy eigenfunctions equidistribute in phase space has remained open. This property was christened \emph{quantum unique ergodicity} by \citet{rudnicksarnak94}, who conjectured that the Laplacian on any compact negatively curved manifold is quantum unique ergodic (QUE). 
%The conjecture was motivated by the fact that classical particles on compact negatively curved manifolds are known to be strongly chaotic.
Although the Rudnick--Sarnak conjecture is still open, it is now known that quantum unique ergodicity is not always valid, even if classical particles are chaotic; see \citet{faurenonnenmacher04}, \citet{faureetal03} and \citet{hassell10}. QUE has been verified in only a handful of cases; in particular for the Hecke orthonormal basis on an arithmetic surface by \citet{lindenstrauss06}, \citet{silbermanvenkatesh07} as well as for modular cusp forms on the modular surface \citet{holowinskysoundararajan10} and \citet{soundararajan10}.  \citet{anantharaman08} made partial progress towards the general Rudnick--Sarnak conjecture by showing that high energy Laplace eigenfunctions on compact negatively curved manifolds cannot concentrate very strongly. For example, they cannot concentrate on a single closed geodesic. For a more comprehensive survey of results on quantum unique ergodicity, see~\citet{sarnak11}. For more on quantum ergodicity and semiclassical chaos, see \citet{zelditch10}.

In spite of the availability of counterexamples to QUE, it is believed that QUE is generically valid for domains with ergodic billiard ball flow (see \citet{sarnak11}). In other words, QUE is expected to be true for almost all ergodic domains. There are at present no results like this.
%, other than the recent work of \cite{bourgadeyau13} who proved QUE for eigenfunctions of random matrices. 

The main result of this paper (Theorem~\ref{quethm}) says that for any Euclidean domain satisfying some very regularity conditions, there exists $S_\e:L^2\to L^2$ with $\|S_\e\|_{L^2\to L^2}\leq \e$ such that the perturbation of the Laplacian (with Dirichlet boundary condition) $-(I+S_\e)\Delta$ is self adjoint and has QUE eigenfunctions.  In other words, Dirichlet Laplacians lie in the closure (in the $H^2\to L^2$ norm topology) of the set of operators with QUE eigenfunctions. If we impose more regularity on the domain, then we can improve the regularity of $S_\e$. The required operator is constructed using a probabilistic method (described briefly in Section \ref{sec:outline}) and it is then shown that this random operator satisfies the required property with probability one. Notice that, although we show that Laplacians are close in the operator norm to QUE operators, this is very far from showing that one can perturb the domain to obtain a QUE Laplacian. Indeed, one should probably not expect such a result to hold for arbitrary domains. 

Our result is closely related to those in \citet{zelditch92,zelditch96,zelditch14}, \citet{maples13}  and \citet{chang15} where it is shown that certain unitary randomizations of eigenfunctions are quantum ergodic. In effect, this shows that $-U_k\Delta U_k^*$ is quantum ergodic for $U_k$ random unitary operator which mixes blocks of eigenfunctions. See Section \ref{sec:compare} for a more detailed comparison of the results.

\section{Results}\label{intro}
\subsection{Definitions}
We start by defining the class of domains to which our results apply. These domains may have boundaries which are quite rough and in particular include all domains where the solution of the Dirichlet problem has the property $u(x)\to 0$ as $x\to \partial\Omega$. 

Take any $d\ge 2$ and let $\om$ be a Borel subset of $\rr^d$. Let $B_t$ be a standard $d$-dimensional Brownian motion, started at some point $x\in \rr^d$. The exit time of $B_t$ from $\om$ is defined as
\begin{align}\label{exit}
\tau_\om := \inf\{t>0: B_t\not\in \om\}\,.
\end{align}
In this paper we will say that $\om$ is a \emph{regular domain} if it is nonempty, bounded, open, connected, and satisfies the following boundary regularity conditions:
\begin{enumerate}[(i)]
%\item $\vol(\partial \om) = 0$, where $\vol$ denotes Lebesgue measure.
\item $\vol(\partial \om) = 0$, where $\partial \om$ is the boundary of $\om$ and $\vol$ denotes $d$ dimensional Lebesgue measure. 
\item For any $x\in \partial \om$, $\pp^x(\tau_\om = 0) = 1$, 
where $\pp^x$ denotes the law of Brownian motion started at $x$ and $\tau_\om$ is the exit time from $\om$. 
\end{enumerate}
Condition (ii) may look strange to someone who unfamiliar with probabilistic potential theory, but it is actually the well-known sharp condition for the existence of solutions to Dirichlet problems on $\om$~\cite[p.~225]{mortersperes10}. A useful sufficient condition for (ii) is that  every point on the boundary satisfies the so-called `Poincar\'e cone condition'~\cite[p.~68]{mortersperes10}. The cone condition stipulates that for every point $x\in \partial \Omega$, there is a cone based at $x$ whose interior lies outside $\Omega$ in a small neighborhood of $x$. Using the Poincar\'e cone condition, it is not difficult to verify that domains with $W^{2,\infty}$ boundaries, considered in \citet{gerardleichtnam93}, satisfy the condition~(ii). However, (i) and (ii) allow more general domains than those with $W^{2,\infty}$ boundary.  For example, any convex open set satisfies the cone condition, irrespective of the smoothness of the boundary. Various kinds of regions with corners, such as polygons, also satisfy the cone condition. 

An example of a domain that {\it does not} satisfy (ii) is the open unit disk in $\rr^2$ minus the interval $(0,1)$. More generally, domains with very sharp cusps at the boundary may not satisfy condition (ii) (see `Lebesgue's thorn' in~\cite[Section 8.4]{mortersperes10}). 

 Henceforth, we will assume that $\om$ is a regular domain and $\dbar$ will denote the closure of~$\om$.

Given any measurable function $f: \dbar\ra \mathbb{C}$, we denote by $\|f\|$  the $L^2(\om)$ norm of $f$. For such $f$ there is a natural probability measure associated with $f$ that has density $|f(x)|^2$ with respect to Lebesgue measure on $\dbar$. We will denote this measure as $\nu_f$. Note that in the definition of $\|f\|$ it does not matter whether we integrate over $\om$ or $\dbar$ since $\vol(\partial \om) = 0$. We will denote the $L^2$ inner product of two functions $f$ and $g$ by $\langle f,g\rangle$. 

Recall that  a sequence of probability measures $\{\mu_n\}_{n\ge 1}$ on $\dbar$ is said to converge weakly to a probability measure $\mu$ if 
\[
\lim_{n\ra\infty}\int_{\dbar} fd\mu_n = \int_{\dbar} fd\mu
\]
for every bounded continuous function $f:\dbar \ra\rr$. A probability measure that will be of particular importance in this paper is the uniform probability measure on $\dbar$. This is simply the restriction of Lebesgue measure to $\dbar$, normalized to have total mass one. %We will denote this measure by $\mu$. 

\subsection{Defect Measures}

For every bounded  sequence of functions $\{f_n\}\in L^2(\re^d)$ with $f_n\underset{L^2}{\rightharpoonup} 0$, we can also associate a family of measures in phase space, $S^*\re^d$ (the cosphere bundle of $\re^d$), called \emph{defect measures}, defined as follows.  Recall the notation $\Psi^m(\re^d)$ for the pseudodifferential operators of order $m$ on $\re^d$, and $S^m_{\text{hom}}(T^*\re^d)$ for smooth functions on $T^*\re^d\setminus\{0\}$ homogeneous of degree $m$ in the fiber variable.   Let $S_{\text{phg}}^m(T^*\re^d)$ denote the associated polyhomogeneous symbol classes. That is, $a\in S^m_{phg}$ if there exists $a_j\in S^{m-j}_{\text{hom}}$ so that 
\begin{equation}
\label{e:expand}\left|\partial_x^\alpha\partial_{\xi}^\beta\left(a(x,\xi)-\sum_{j=0}^{N-1}a_j(x,\xi)\right)\right|\leq C_{\alpha\beta}(1+|\xi|^2)^{m-N-|\beta|},\quad |\xi|\geq 1.
\end{equation}
(See \cite{HOV3} for more details.)
We sometimes write $\Psi(\re^d)$, $S_{\text{hom}}(T^*\re^d)$, and $S_{phg}(T^*\re^d)$ for $\Psi^0(\re^d)$, $S^0_{\text{hom}}(T^*\re^d)$ and $S^0_{phg}(T^*\re^d)$ respectively. We also sometimes omit the $\re^d$ or $T^*\re^d$ when the relevant space is clear from context.

Let 
\[
\sigma:\Psi^m\to S^m_{\text{hom}}
\]
be the principal symbol map on $\Psi^m(\re^d)$. For $b\in S^m_{phg}$, we write $b(x,D)\in \Psi^m$ for a quantization of $b$ and observe that 
$$\sigma(b(x,D))=b_0(x,\xi)$$
where $b_0\in S^m_{\text{hom}}$ is the first term in the expansion \eqref{e:expand} for $b$.

Let $\chi \in \Cc(\re^d)$ have $\chi\equiv 1$ in a neighborhood of 0. For $a\in \Cc(S^*\re^d)$, let 
\[
\tilde{a}(x,\xi)=a(x,\xi/|\xi|)(1-\chi(\xi))\,.
\]
Then $\tilde{a}\in S_{phg}$.
Define the distribution $\mu_n\in \mc{D}'(S^*\re^d)$ by
\[
\mu_n(a)=\la \tilde{a}(x,D)f_n,f_n\ral
\]
where $\la \cdot,\cdot\ral$ denotes the inner product in $L^2(\re^d)$ and $D:=-i\partial$ is $-i$ times the gradient operator. Not that the weak convergence of $f_n$ to zero implies that for every subsequence of $\{\mu_n\}_{n\ge 1}$ there is a further subsequence that converges in the $\mc{D}'(S^*\re^d)$ topology.  Moreover, it can be shown that every limit point $\mu$ of $\{\mu_n\}_{n\ge 1}$ in the $\mc{D}'(S^*\re^d)$ topology is a positive radon measure, with the property that there exists a subsequence $\{f_{n_k}\}_{k\ge 1}$  so that for all $A\in \Psi(\re^d)$
\[
\la Af_{n_{k}},f_{n_{k}}\ral\to \int_{S^*\re^d}\sigma(A)d\mu\,.
\]
(See for example \cite{Burq} or \cite{gerard91}.) The set of such limit points $\mu$ is denoted by $\mc{M}(\{f_n\}_{n\ge 1})$ and is called the \emph{set of defect measures associated to the family $\{f_n\}_{n\ge 1}$}. We will  write $\mc{M}(f_n)$ instead of $\mc{M}(\{f_n\}_{n\ge 1})$ to simplify notation. Note that while $\mu_n$ depends on the choice of quantization procedure used to define $\tilde{a}(x,D)$ and the function $\chi$, the set $\mc{M}(f_n)$ is independent of such choices. 

\subsection{QUE operators}

If $H$ is a linear operator from a dense subspace of $L^2(\dbar)$ into $L^2(\dbar)$, we will say that a function $f$ belonging to the domain of $H$ is an eigenfunction of $H$ with eigenvalue $\lambda\in \mathbb{C}$ if $f\ne 0$ and $Hf = \lambda f$. We will say that an eigenfunction $f$ is normalized if $\|f\|=1$. 

\begin{defn}\label{quedef}
Let $H$ be a linear operator from some dense subspace of $L^2(\dbar)$ into $L^2(\dbar)$ having compact resolvent. We say that $H$ is QUE if for any sequence of normalized eigenfunctions $\{f_n\}_{n\ge 1}$ of $H$, 
\begin{equation}
\label{queeqn}
\mc{M}(1_{\dbar}f_n)= \left\{\frac{1}{\vol(\om)}1_{\dbar}dxdS(\xi)\right\}
\end{equation} where $S$ is the normalized surface measure on $S^{d-1}$.
\end{defn}
\noindent In particular, notice that if \eqref{queeqn} holds then for all $A\in \Psi^0(\re^d)$,
$$\la A1_{\dbar}f_n,1_{\dbar}f_n\ral\to \frac{1}{\vol(\om)}\int_{S^*\re^d}\sigma(A)1_{\dbar}dxdS(\xi)$$ 
and hence that $\nu_{f_n}$ converges weakly as a measure to the uniform probability distribution on $\om$. 
With this in mind, we define the weaker notion of equidistribution as follows.
\begin{defn}\label{quedef2}
Let $H$ be a linear operator from some dense subspace of $L^2(\dbar)$ into $L^2(\dbar)$ having compact resolvent. We say that $H$ is uniquely equidistributed if for $\{f_n\}_{n\ge 1}$ any sequence normalized eigenfunctions of $H$, $\nu_{f_n}$ converges weakly to the uniform probability distribution on $\om$.
\end{defn}
\subsection{The main result}
Let $-\Delta $ be the Dirichlet Laplacian on $\om$  with domain $\mfd$ (defined in Section \ref{prelim}). The following theorem is the main result of this paper. 
\begin{thm}\label{quethm}
Let $\om$ be a regular domain. Then for any $\ep>0$, there exists a linear operator $S_\ep:L^2(\dbar)\ra L^2(\dbar)$ such that: 
\begin{enumerate}
\item[\textup{(i)}] $\|S_\ep\|_{L^2\to L^2}\le \ep$.  
\item[\textup{(ii)}]  $-(I+S_\ep)\Delta$ is a positive operator on $L^2(\dbar)$ with domain $\mfd$
\item[\textup{(iii)}] $-(I+S_\ep)\Delta$ is QUE in the sense of Definition \ref{quedef}.
\end{enumerate}
If $\om$ has $C^\infty$ boundary, then for all $\gamma<1$, there exist such an $S_\ep:L^2(\dbar)\to H^\gamma(\dbar)$ with $\|S_\e\|_{L^2\to H^\gamma}\le \ep$. Moreover, if $\om$ has smooth boundary and the set of periodic billiards trajectories has measure zero (see Section \ref{billiards}), then this holds for $\gamma\leq 1$. 
\end{thm}
It would be interesting to see if a different version of this theorem can be proved, where instead of perturbing the Laplacian, it is the domain $\om$ that is perturbed. Alternatively, one can try to perturb the Laplacian by some explicit kernel rather than saying that `there exists $S_\ep$'. Yet another possible improvement would be to show that a generic perturbation, rather than a specific one, results in an operator with QUE eigenfunctions. Indeed, the proof of Theorem \ref{quethm} gets quite close to this goal. 

\subsection{Additional results}
The techniques of this paper yield the following version of the local Weyl law for regular domains.
\begin{thm}
\label{localWeyl}
Suppose that $\om\subset \re^d$ is a regular domain, where regularity is defined at the beginning of this section. Let $\{(u_j,\lambda_j^2)\}_{j\ge1}$ be a complete orthonormal basis of eigenfunctions of the Dirichlet Laplacian on  $\om$. Then for $A\in \Psi(\re^d)$ with $\sigma(A)$ supported in a compact subset of $\om$ and any $E>1$,
$$\sum_{\lambda_j\in [\lambda,\lambda E]}\la A1_{\dbar}u_j,1_{\dbar}u_j\ral =\frac{\lambda^d}{(2\pi )^d}\iint_{1\leq |\xi|\leq E}\sigma(A)1_{\dbar}dxd\xi+\o{}(\lambda^{d}).$$ 
\end{thm}

In order to state the next theorem, we need the following definition.
\begin{defn}\label{assume}
Let $C_0(S^*\om)$ be the set of continuous functions on $S^*\re^d$ that vanish on $(\re^d\setminus \om)\times S^{d-1}$ with the sup-norm topology. 
Let $\alpha:\re_+\to \re_+$ be nonincreasing. Let $\{(u_j,\lambda_j^2)\}_{j\ge 1}$ be a complete orthonormal basis of eigenfunctions of the Dirichlet Laplacian on $\om$. Suppose that there exists $\mc{A}\subset \Psi(\re^d)$ so that for all $A\in \mc{A}$, $\sigma(A)$ is supported compactly inside $S^*\om$, the set
\[
\sigma(\mc{A}):=\{\sigma(A)|_{S^*\re^d}: A\in \mc{A}\}
\]
is dense in $C_0(S^*\om)$, and for each $A\in\mc{A}\subset  \Psi(\re^d)$,
\begin{align*}
\sum_{\lambda_j\in [\lambda,\lambda(1+\alpha(\lambda))]}&\la A1_{\dbar}u_j, 1_{\dbar}u_j\ral \\
&=\frac{\lambda^d}{(2\pi )^d}\iint_{1\leq |\xi|\leq 1+\alpha(\lambda)}\sigma(A)1_{\dbar}dxd\xi+\o{}(\alpha(\lambda)\lambda^{d}).
\end{align*}
Then we say that the domain $\om$ is \emph{average quantum ergodic (AQE) at scale~$\alpha$}. 
\end{defn}
Theorem \ref{localWeyl} implies that regular domains $\om$ are AQE at scale $E$ for any $E>0$. In Section~\ref{sec:localWeyl}, we recall Weyl laws holding on domains with $C^\infty$ boundaries which imply that these domains are AQE at scale $\alpha(\lambda)=\lambda^{-\gamma}$ for some $\gamma>0$. For $\gamma\in [0,2]$, let $\mfd^\gamma$ denote the complex interpolation space $(L^2(\om),\mfd)_{\gamma/2}$. Then the following theorem implies Theorem \ref{quethm}.
\begin{thm}\label{quethm2}
Suppose that $\om$ is a regular domain that is AQE at scale $\alpha(\lambda)=\lambda ^{-\gamma}$ for some $1\geq \gamma\geq 0$. Then for any $\ep>0$, there exists a linear operator $S_\ep:L^2(\dbar)\ra \mfd^\gamma$ such that: 
\begin{enumerate}
\item[\textup{(i)}] $\|S_\ep\|_{L^2\to \mfd^\gamma}\le \ep$.  
\item[\textup{(ii)}]  $-(I+S_\ep)\Delta$ is a positive operator on $\mfd$ with compact resolvent
\item[\textup{(iii)}] $-(I+S_\ep)\Delta$ is QUE in the sense of Definition \ref{quedef}.
\end{enumerate}
\end{thm}

A consequence of Theorem \ref{quethm} is that $-\Delta$ has a sequence of `quasimodes' that are equidistributed in the limit. Moreover, when $\om$ is AQE at some scale $\alpha(\lambda)=\o{}(1)$, then there is a full orthonormal basis of (slightly weaker) quasimodes that are QUE. This is the content of the following corollary.

\begin{cor}\label{almostcor}
Let all notation be as in Theorem \ref{quethm}. Suppose that $\om$ is AQE at scale $\alpha(\lambda)=\lambda^{-\gamma}$ for some $\gamma\geq 0$. Then
\begin{itemize}
\item[(i)] there is a sequence of functions $\{f_n\}_{n\ge 1}$ belonging to $\mfd$ and a sequences of positive real numbers $\{\alpha_n\}_{n\ge 1}$, such that $\|f_n\|=1$, $\alpha_n \ra\infty$ and
$
(-\alpha_n^{-2} \Delta -1) f_n=\o{L^2}(\alpha_n^{-\gamma}) 
$
and
$$\mc{M}(f_n)=\left\{\frac{1}{\vol(\om)}1_{\dbar}dxd\sigma(\xi)\right\}.$$
\item[(ii)]  there is an orthonormal basis of $L^2(\dbar)$, $\{g_n\}_{n\ge 1}$ belonging to $\mfd$ and a sequences of positive real numbers $\{\beta_n\}_{n\ge 1}$, such that $\beta_n \ra\infty$ and
$
(-\beta_n^{-2} \Delta -1) g_n=\O{L^2}(\beta_n^{-\gamma}).
$
and $$\mc{M}(g_n)=\left\{\frac{1}{\vol(\om)}1_{\dbar}dxd\sigma(\xi)\right\}.$$
\end{itemize}
\end{cor}

\begin{remark}
Note that up to this point, all results apply equally well to compact manifolds with or without boundary, but we chose to present them for the case of $\om\Subset \re^d$ for concreteness.
\end{remark}

\subsection{Improvements on compact manifolds without boundary}
Together with the analog of Theorem \ref{quethm}, a stronger version of the Weyl law valid on compact manifolds without boundary (see Section \ref{sec:localWeyl}), implies the following corollary.
\begin{cor}\label{almostcor2}
Let $(M,g)$ be a compact Riemannian manifold without boundary so that the set of closed geodesics has measure 0. Then there is an orthonormal basis of $L^2(\dbar)$, $\{f_n\}_{n\ge 1}$, belonging to $C^\infty(M)$ and a sequence of positive real numbers $\{\alpha_n\}_{n\ge 1}$ such that $\alpha_n \ra\infty$, 
\[
(-\alpha_n^{-2} \Delta_g -1)f_n =\o{L^2}(\alpha_n^{-1})\, ,
\]
and $\nu_{f_n}\to \frac{1}{\vol(M)}dx$. That is, $f_n$ are uniquely equidistributed in the sense of Definition \ref{quedef2}.
\end{cor}
Unfortunately, the authors were unable to prove a version of Corollary \ref{almostcor2} where the basis of quasimodes is QUE rather than uniquely distributed. This is because the remainder in the strong version of the local Weyl law (see Theorem \ref{thm:weylStrong}) may depend on derivatives of the symbol in $\xi$.

\begin{remark}
Notice that if $M$ has ergodic geodesic flow, then the set of periodic geodesics has measure zero and hence Corollary \ref{almostcor2} applies and there is an orthonormal basis of $\o{L^2}(\alpha^{-1})$ quasimodes that are equidistributed. In particular, being $\o{L^2}(\alpha^{-1})$ quasimodes implies that these functions respect the dynamics at the level of defect measures, that is, defect measures associated to the family of quasimodes are invariant under the geodesic flow. See \cite{Burq} or \cite[Chapter 5]{EZB} for details.
\end{remark}

\subsection{Comparison with previous results}
\label{sec:compare}
One can view the results here as a companion to those in \cite{zelditch92,zelditch96,zelditch14} and \cite{maples13}. In these papers, the authors work on a compact manifold $M$ and fix a basis of eigenfunctions of the Laplacian, $\{u_n\}_{n=1}^\infty$. Their results then show that for almost every block diagonal (in the orthonormal basis $u_n$) unitary operator 
$$U=\oplus_{k=1}^\infty U_k$$
(with respect to the product Haar measure) 
such that for all $k$, $ \dim \ran U_k<\infty$ and $\dim \ran U_k\to \infty$ at least polynomially in $k$, the basis $\{Uu_n\}_{n=1}^\infty$ has 
$$\mc{M}(Uu_n)=\left\{\frac{1}{\vol(M)}dxdS(\xi)\right\}.$$

One reformulation considers a certain basis of eigenfunctions for the operator $-U\Delta U^*$. By taking $U_k$ close to the identity, we may write 
$$\tilde{P}:=-U\Delta U^*=-(I+\tilde{S})\Delta$$ 
where $\tilde{S}$ is small in $L^2\to L^2$ norm. However, $\tilde{P}$ may not be QUE if there is high multiplicity in the spectrum of $-\Delta$.

One can think of the results in the present paper as replacing the $U_k$ by some nearly unitary operator. By choosing these operators carefully, and employing the Hanson--Wright inequality in place of the law of large numbers, we are able to use smaller windows than those in previous work. This allows us to prove that the perturbation is regularizing under various conditions, and to show that the resulting operator is QUE.

\subsection{Outline of the proof and organization of the paper}
\label{sec:outline}
In order to prove Theorem \ref{quethm}, we show that a local Weyl law with a certain window implies the existence of the desired perturbation $S_\e$. The local Weyl law essentially says that when averaged over a certain size window, say $\lambda^{-\gamma}$, the matrix elements $\langle Au_k,u_k\rangle$ behave as though the eigenfunctions were uniquely ergodic. In Section \ref{rotationsec} we give a rigorous meaning to this statement. In particular, we use a modern version of the Hanson--Wright inequality from \cite{rudelsonvershynin13} (see \cite{hansonwright71} for the original) to show that random rotations (with respect to Haar measure) of small groups of eigenfunctions are uniquely ergodic. Here, the size of the group allowed depends on the remainder in the local Weyl law. Thus, the smaller the remainder, the smaller the required group of eigenfunctions.

In Section \ref{last}, we obtain the perturbation, $S_\e$. In order to do this, we make a two scale partition of the eigenvalues, $\lambda_i^2$, of the Laplacian. In particular, we divide the eigenvalues of the Laplacian into 
\begin{align*} L_{n,j}&:=\left\{\lambda_i: \left(1+\frac{\e j}{\lceil(1+\e)^{n\gamma}\rceil}\right) \leq \frac{\lambda_{i}}{(1+\e)^n}<\left(1+\frac{\e (j+1)}{\lceil(1+\e)^{n\gamma}\rceil}\right)\right\},\\
 &\qquad \qquad 0\leq n,\quad 0\leq j\leq \lceil(1+\e)^{n\gamma}\rceil-1
 \end{align*}
where $\gamma$ is determined by the remainder in the local Weyl law.
For each $L_{n,j}$ we then make a random rotation of the corresponding eigenfunctions and perturb the eigenvalues, $\lambda_i\to \lambda'_i$ so that each new eigenvalue, $\lambda_i'$ is simple and lies in $L_{n,j}$.

 Because of the fact that random rotations of eigenfunctions on the scale $\lambda^{-\gamma}$ are QUE, this results in an operator that is almost surely QUE. The regularizing nature of the perturbation results from the second scale in $L_{n,j}$. That is, the fact the eigenfunctions with eigenvalue similar to $(1+\e)^n$ are mixed only with those whose eigenvalues are at a distance $(1+\e)^{n(1-\gamma)}\e$. In particular, the larger $\gamma$, the more regularizing the perturbation. %As mentioned before, the technique of using random rotations of small collections of eigenfunctions has similarities with ideas from \cite{zelditch14} and \cite{maples13}. 
 
 In order to prove Theorem \ref{quethm},  we need to prove the local Weyl law for regular domains (Theorem \ref{localWeyl}), but we postpone this proof until Appendix \ref{ptwisesec}. The key ingredient here is to compare the heat trace for the Dirichlet Laplacian on $\Omega$ with the heat trace for the Laplacian on $\re^d$ as in \cite{gerardleichtnam93, dodziuk81}. Let $k(t,x,y)$ and $k_D(t,x,y)$ be respectively the kernels of $e^{t\Delta}$ and $e^{t\Delta_D}$, where $\Delta$ is the free Laplacian and $\Delta_D$ the Dirichlet Laplacian. The key estimate in proving Theorem \ref{localWeyl} is 
$$|\partial_x^\alpha(k(t,x,y)-k_D(t,x,y))|\leq C_\delta t^{-N_\alpha}e^{-c_\delta/t},\quad\quad d(x,\partial\Omega)>\delta.$$
\begin{remark}
The work of \citet{LiStroh} extends this type of estimate to the heat kernel for general self-adjoint extensions of the Laplacian and gives explicit constants independently of the extension.
\end{remark}

We prove this estimate using the relationship between killed Brownian motion on $\Omega$ with the Dirichlet heat Laplacian together with the fact that Brownian motion has independent increments. Because of this approach, we are able to complete the proof on domains which are only regular.

The paper is organized as follows. Section \ref{sec:prelim} recalls local Weyl laws valid for domains with smooth boundary and manifolds without boundary, the functional analytic definition of the Dirichlet Laplacian, and some geometric preliminaries. Section \ref{rotationsec} presents the results on random rotations of eigenfunctions. Section \ref{last} finishes the proof of Theorems \ref{quethm}, \ref{quethm2} and Corollary \ref{almostcor}.  Section~\ref{last2} contains the adjustments necessary to obtain the improvements on manifolds without boundary, in particular proving Corollary~\ref{almostcor2}. Finally, Appendix \ref{ptwisesec} contains the proof of Theorem \ref{localWeyl}.

\section{Preliminaries}
\label{sec:prelim}
Throughout the paper, we will adopt the notation that $C$ denotes any positive constant that may depend only on the set $\om$, the dimension $d$, and nothing else. The value of $C$ may change from line to line. In case we need to deal with multiple constants, they will be denoted by $C_1,C_2,\ldots$. From this point forward we will assume that 
\[
\vol(\om) = 1\,. 
\]
This does not result in any loss of generality since we may always rescale $\om$ with positive volume to have unit volume. 
\subsection{Local Weyl Laws}
\label{sec:localWeyl}
We first recall some now classical local Weyl laws for domains $\om$ more regular than those in Theorem \ref{localWeyl}.  In this setting, we have the following version of the local Weyl law~\cite{DuiGui, safarovVassiliev}.
\begin{thm}
Suppose that $\om$ has $C^\infty$ boundary. Let $\{(u_j,\lambda_j^2)\}_{j\ge 1}$ be a complete orthonormal basis of eigenfunctions of the (Dirichlet) Laplacian on $\om$. Then for $A\in \Psi(\re^d)$ with $A$ having kernel supported in a compact subset of $\om\times \om$,  
$$\sum_{\lambda_j\in [\lambda,\lambda E]}\la A1_{\dbar}u_j,1_{\dbar}u_j\ral =\frac{\lambda^d}{(2\pi )^d}\iint_{1\leq |\xi|\leq E}\sigma(A)1_{\dbar}dxd\xi+\O{}(\lambda^{d-1}).$$
In particular, $\om$ is AQE at scale $\lambda^{-\gamma}$ for any $\gamma<1$.
Moreover if the set of closed trajectories for the billiard flow (see Section~\ref{billiards}) has measure zero, then
\begin{align*}
\sum_{\lambda_j\in [\lambda,\lambda(1+\lambda^{-1})]}&\la A1_{\dbar}u_j,1_{\dbar}u_j\ral \\
&=\frac{\lambda^d}{(2\pi )^d}\iint_{1\leq |\xi|\leq 1+\lambda^{-1}}\sigma(A)1_{\dbar}dxd\xi+\o{}(\lambda^{d-1}).
\end{align*}
In particular, $\om$ is AQE at scale $\lambda^{-1}$.
\end{thm}

\subsection{Manifolds without boundary}
Let $(M,g)$ be a smooth compact Riemannian manifold without boundary (i.e. a smooth manifold with smooth metric). Then the Laplace operator is given in local coordinates by 
$$-\Delta_g:=\frac{1}{\sqrt{|g|}}\partial_i(\sqrt{|g|}g^{ij}\partial j)$$
where $|g|=\det g_{ij}$ and $g(\partial_{x_i},\partial_{x_j})=g_{ij}$ with inverse $g^{ij}$. The operator $-\Delta_g$ has domain $H^2(M)$ and is invertible as an operator $L^2_m(M)\to H^2_m(M)$ where $B_m(M)$ is the set of functions in $B$ with $0$ mean. 
In this setting, we have the following version of the pointwise Weyl law~\cite{safarovVassiliev}.
\begin{thm}
\label{thm:weylStrong}
Let $\{(u_j,\lambda_j^2)\}_{j\ge 1}$ be the eigenfunctions of $-\Delta_g$. Then
$$\sum_{\lambda_j\leq \lambda} |u_j(x)|^2=\frac{\lambda^d}{(2\pi)^d\vol(M)}\vol(B_d)+\O{}(\lambda^{d-1})$$
where $B_d$ denotes the unit ball in $\re^d$.
If the set of closed geodesics has zero measure, then $\O{}(\lambda^{d-1})$ can be replaced by $\o{}(\lambda^{d-1}).$ Moreover, the asymptotics are uniform for $x\in M$. 
\end{thm} 
Theorem \ref{thm:weylStrong} provides estimates uniform in $x$ that are used to prove Corollary \ref{almostcor2}.

\subsection{Functional Analysis}
\label{prelim}
Recall our convention that $\|f\|$ denotes the $L^2$ norm of a function $f$ and $\la f,g\ral$ denotes the $L^2$ inner product of $f$ and $g$. 

Let $\Omega\subset \re^d$ a bounded open set. We now recall the definition of the Dirichlet Laplacian as a self adjoint unbounded operator on $L^2(\om)$. Let $H_0^1(\om)$ denote the closure of $\Cc(\om)$ with respect to the $H^1$ norm where for $k\in \mathbb{N}$,
$$\|u\|^2_{H^k(\om)}:=\sum_{|\alpha|\leq k}\|\partial^\alpha u\|^2.$$
Here for a multiindex $\alpha\in \mathbb{N}^d$, 
$$\partial^\alpha=\partial_{x_1}^{\alpha_1}\partial_{x_2}^{\alpha_2}\dots\partial_{x_d}^{\alpha_d},\quad\quad |\alpha|=\alpha_1+\alpha_2+\dots \alpha_d.$$
Then $H_0^1(\om)$ is a Hilbert space with inner product 
$$(u,v)=\la u,v\ral+\la \nabla u,\nabla v\ral.$$
Define the quadratic form $Q:H_0^1(\om)\times H_0^1(\om)\to \mathbb{C}$ by
$$Q(u,v)=\la \nabla u, \nabla v\ral.$$
Then $Q$ is a symmetric, densely defined quadratic form and for $u,v\in H_0^1(\om)$,
$$|Q(u,v)|\leq C\|u\|_{H^1(\om)}\|v\|_{H^1(\om)},\quad\quad c\|u\|_{H^1(\om)}^2\leq Q(u,u)+C\|u\|^2.$$
Therefore by \cite[Theorem VIII.15]{Reed}, $Q$ defines a unique self-adjoint operator $-\Delta$ with domain
\[
\mfd := \{u\in H_0^1: Q(u,w)\leq C_u\|w\| \text{ for all } w\in H_0^1(\om)\}\,.
\]
This operator is called the \emph{Dirichlet Laplacian}.  Let $E_\mu$ denote the resolution of the identity for $-\Delta$, i.e. $E_\mu=1_{(-\infty, \mu]}(-\Delta)$. Then the complex interpolation space between $L^2$ and $\mfd$ is given by
$$(L^2,\mfd)_\theta:=\Big\{f\in L^2\mid \int \langle \mu\rangle ^\theta dE_\mu f\in L^2\Big\},\qquad\langle \mu\rangle :=(1+|\mu|^2)^{1/2}.$$ 
 We recall that
\begin{lmm}
\label{lmmdomain}
Suppose that $\om$ has $C^2$ boundary. Then $\mfd=H_0^1(\om)\cap H^2(\om)$ and in particular $(L^2,\mfd)_\theta\subset H^{2\theta}(\om)$. 
\end{lmm}

\subsection{The billiard flow}
\label{billiards}
Let $\om$ be a domain with $C^\infty$ boundary. We now define the billiard flow. Let $S^*\re^d$ be the unit sphere bundle of $\re^d$. We write 
$$ S^*\re^d|_{\partial \om}=\partial \om_+\sqcup\partial \om_-\sqcup\partial \om_0\,\,$$
where $(x,\xi)\in \partial \om_+$ if $\xi$ is pointing out of $ \om$, $(x,\xi)\in\partial \om_-$ if it points inward, and $(x,\xi)\in\partial \om_0$ if $(x,\xi)\in S^*\partial \om$. The points $(x,\xi)\in \partial \om_0$ are called \emph{glancing} points. Let $B^*\partial \om$ be the unit coball bundle of $\partial \om$, i.e. 
$$B^*\partial\om =\{(x,\xi)\in T^*\partial \om \mid |\xi|_g<1\}$$
 and denote by $\pi_{\pm}:\partial \om_{\pm}\to B^*\partial \om$ and $\pi:S^*\re^d|_{\partial \om}\to \overline{B^*\partial \om}$ the canonical projections onto $\overline{B^*\partial \om}$. Then the maps $\pi_{\pm}$ are invertible. Finally, write 
$$ t_0(x,\xi)=\inf\{t>0:\exp_t(x,\xi)\in T^*\re^d|_{\partial \om}\}\,\,$$
where $\exp_t(x,\xi)$ denotes the lift of the geodesic flow to the cotangent bundle. That is, $t_0$ is the first positive time at which the geodesic starting at $(x,\xi)$ intersects $\partial \om$.

We define the billiard flow as in \cite[Appendix A]{DyZw}. Fix $(x,\xi)\in S^*\re^d\setminus \partial \om_0$ and denote $t_0=t_0(x,\xi)$. Then since $\partial\om$ is $C^\infty$ and $(x,\xi)\notin \partial\Omega_0$, $t_0\in (0,\infty]$. We assume now that $t_0<\infty$. If $\exp_{t_0}(x,\xi)\in \partial \om_0$, then the billiard flow cannot be continued past $t_0$. Otherwise there are two cases: $\exp_{t_0}(x,\xi)\in \partial \om_+$ or $\exp_{t_0}(x,\xi)\in \partial \om_-$. We let 
$$(x_0,\xi_0)=\begin{cases}
\pi_-^{-1}(\pi_+(\exp_{t_0}(x,\xi)))\in \partial \om_-\,,&\text{if }\exp_{t_0}(x,\xi)\in \partial \om_+\\
\pi_+^{-1}(\pi_-(\exp_{t_0}(x,\xi)))\in \partial \om_+\,,&\text{if }\exp_{t_0}(x,\xi)\in \partial \om_-.
\end{cases}$$
That is, $(x_0,\xi_0)$ is the reflection of $\exp_{t_0}(x,\xi)$ along the normal bundle of $\partial\om$ through $T_x^*\partial\om$.  
We then define $\varphi_t(x,\xi)$, the \emph{billiard flow}, inductively by putting 
\[
\varphi_t(x,\xi)=\begin{cases}\exp_t(x,\xi)&0\leq t<t_0,\\
\varphi_{t-t_0}(x_0,\xi_0)&t\geq t_0.
\end{cases}
\]
We say that the trajectory starting at $(x,\xi)\in S^*\re^d$ is \emph{periodic} if there exists $t>0$ such that $\varphi_t(x,\xi)=(x,\xi).$ 

\subsection{Probabilistic Notation}
We now introduce a few notations from probability. Recall that $\pp(A)$ denotes the probability of the event $A$ and $\ee(X)$ denotes the expected value of the random variable $X$. Finally, $\ee(X;A)$ denotes the expectation of the random variable $X$ conditioned on~$A$.

\section{Concentration of random rotations}\label{rotationsec}
Let $u_1,\ldots, u_n$ be an orthonormal set of real valued functions belonging to $L^2(\dbar)$. Let $Q$ be an $n\times n$ Haar-distributed random orthogonal matrix. Let $q_{ij}$ denote the  $(i,j)^{\textup{th}}$ entry of $Q$. Define a new set of functions $v_1,\ldots, v_n$ as 
\[
v_i(x) := \sum_{j=1}^n q_{ij} u_j(x)\,.
\]
Then $v_1,\ldots, v_n$ are also orthonormal, since
\begin{align*}
\la v_i,v_j\ral &= \la \sum_{k,l=1}^n q_{ik}q_{jl}  u_k,u_l\ral= \sum_{k,l=1}^n q_{ik} q_{jl} \la u_k,u_l\ral\\
&= \sum_{k=1}^n q_{ik}q_{jk} = 
\begin{cases}
1 &\text{ if } i=j\, ,\\
0 &\text{ otherwise.}
\end{cases}
\end{align*}
We will refer to $v_1,\ldots, v_n$ as a \emph{random rotation} of $u_1,\ldots, u_n$.  The goal of this section is to prove the following concentration inequality for random rotations.
\begin{thm}\label{rotthm}
Let $u_i$ and $v_i$ be as above. Let $A:L^2(\om )\to L^2(\om )$ be a bounded operator. 
Then for any $1\le i\le n$ and any $t>0$,
\begin{align*}
\pp\biggl(\biggl|\la Av_i,v_i\ral - \frac{1}{n}\sum_{i=1}^n\la Au_i,u_i\ral\biggr| \ge t\biggr) &\le C_1 \exp\bigl(-C_2(\|A\|)  \min\{t^2, \, t\} n\bigr)\,,
\end{align*}
where $C_1$ depends only of $d$ and $\om $, and $C_2(\|A\|)$ depends on $d$, $\om $ and the operator norm, $\|A\|$. 
\end{thm}
The key ingredient in the proof of Theorem \ref{rotthm} is the Hanson--Wright inequality~\cite{hansonwright71} for quadratic forms of sub-Gaussian random variables. The original form of the Hanson--Wright inequality does not suffice for our objective. Instead, the following modern version of the inequality, proved recently by \citet{rudelsonvershynin13}, is the one that we will use. 

The reason why $\la Av_i, v_i\ral$ is concentrated around its mean is that it can be expressed approximately as a quadratic form of i.i.d.~Gaussian random variables, and the eigenvalues of the matrix defining this quadratic form are roughly of equal size. The spectral decomposition then implies that this quadratic form can be written as a linear combination of squares of i.i.d.~Gaussian random variables, where the coefficients are roughly of equal size. The details are worked out below.

Define the $\psi_2$ norm of a random variable $X$ as 
\[
\|X\|_{\psi_2} := \sup_{p\ge 1} p^{-1/2}(\ee|X|^p)^{1/p}\,.
\]
The random variable $X$ is called sub-Gaussian if its $\psi_2$ norm is finite. In particular, Gaussian random variables have this property. 

Let $M= (m_{ij})_{1\le i,j\le n}$ be a square matrix with real entries. The Hilbert--Schmidt norm of $M$ is defined as
\[
\|M\|_{\textup{HS}} := \biggl(\sum_{i,j=1}^n m_{ij}^2\biggr)^{1/2}\,,
\]
and the operator norm of $M$ is defined as
\[
\|M\| := \sup_{x\in \rr^n,\, \|x\|=1}\|Mx\|\,,
\]
where the norm on the right side is the Euclidean norm on $\rr^n$. Rudelson and Vershynin's version of the Hanson--Wright inequality states that if $X_1,\ldots, X_n$ are independent random variables with mean zero and $\psi_2$ norms bounded by some constant $K$, and 
\[
R := \sum_{i,j=1}^n m_{ij} X_iX_j\,,
\]
then for any $t\ge 0$,
\begin{equation}\label{hs}
\pp(|R - \ee(R)|\ge t) \le 2\,\exp\biggl(-C \min\biggl\{\frac{t^2}{K^4\|M\|_{\textup{HS}}^2}, \frac{t}{K^2\|M\|}\biggr\}\biggr)\,,
\end{equation}
where $C$ is a positive universal constant.% For us, the important feature of the standard Gaussian measure is that its $\psi_2$ norm is finite.
\begin{proof}[Proof of Theorem \ref{rotthm}]
Fix $1\le i\le n$. Define 
\[
A_i := \la Av_i,v_i\ral,\quad\quad
B := \frac{1}{n}\sum_{i=1}^n\la Au_i,u_i\ral.
\]
Notice that for each $j$,
\[
\sum_{i=1}^n q_{ij}^2 = 1
\]
and for each $j\ne k$,
\[
\sum_{i=1}^n q_{ij} q_{jk}=0.
\]
Note that the distribution of $Q$ remains invariant under arbitrary permutations of rows. Therefore, for any $j$ and $k$, $\ee(q_{ij}q_{ik})$ is the same for each $i$. Thus, the above identities imply that 
\[
\ee(q_{ij} q_{ik}) =
\begin{cases}
1/n &\text{ if } j=k\, ,\\
0 &\text{ otherwise.}
\end{cases}
\]
Therefore 
\begin{equation}\label{hs1}\begin{aligned}\mathbb{E}(\la Av_i,v_i\ral)&=\sum_{jk}\mathbb{E}(\la A q_{ij}u_j,q_{ik}u_k\ral)=\sum_{jk}\mathbb{E}(q_{ik}q_{ij})\la Au_j,u_k\ral \\
&=\frac{1}{n}\sum_{j=1}^n\la Au_j,u_j\ral=B.\end{aligned}\end{equation}
Let $q_i$ be the vector whose $j^{\textup{th}}$ component is $q_{ij}$. Since $Q$ is a Haar-distributed random orthogonal matrix, symmetry considerations imply that $q_i$ is uniformly distributed on the unit sphere $S^{n-1}$. Now recall that if $z$ is an $n$-dimensional standard Gaussian random vector, then $z/\|z\|$ is uniformly distributed on $S^{n-1}$, and is independent of $\|z\|$. Therefore if $r_i$ is a random variable that has the same distribution as $\|z\|$ and is independent of $q_i$, then the vector $r_i q_i$ is a standard Gaussian random vector. Let $w_{ij} := r_iq_{ij}$, so that $w_{i1},\ldots, w_{in}$ are i.i.d.~standard Gaussian random variables. 
Let $H$ be the matrix with $(j,k)^{\text{th}}$ entry 
$$h_{jk}=\la Au_j,u_k\ral$$
so that 
$$A_i=\sum_{j,k} q_{ij}q_{ik}h_{jk}.$$
Then define
\[
A_i' := r_i^2 A_i =\sum_{j,k=1}^n w_{ij} w_{ik} h_{jk}\,.
\]
Note that $H$ can also be written as $H=\Pi A\Pi$, where $\Pi$ denotes orthogonal projection onto $\text{span}\{u_j: 1\le j\le n\}$. Therefore
$$\|H\|\leq \|A\|$$
and
\begin{align*}
\|H\|_{\text{HS}}&=\sqrt{\tr(H^*H)}= \sqrt{\tr(\Pi^*A^*A\Pi)}\\
&= \sqrt{\tr(A^*A \Pi \Pi^*)} = \sqrt{\tr(A^*A)} = \|A\|_{\text{HS}}\leq \|A\|\sqrt{n}.
\end{align*}
Therefore by the Hanson--Wright inequality \eqref{hs}, with $X_j=w_{ij}$ and $m_{ij}=h_{ij}$ gives
\begin{align}\label{hs2}
\pp(|A_i'-\ee(A_i')|\ge t) \le 2\exp\biggl(-C(\|A\|)\min\biggl\{\frac{t^2}{n},\, t\biggr\}\biggr)\,,
\end{align}
where $C(\|A\|)=\min (\|A\|^{-2},\|A\|)$. Again note that by the Hanson--Wright inequality and the fact that $\ee(r_i^2)=n$,
\begin{align}\label{hs3}
\pp(|r_i^2-n|\ge t)\le 2\exp\biggl(-C\min\biggl\{\frac{t^2}{n},\, t\biggr\}\biggr)\,.
\end{align}
Next, note that
\begin{align}
|A_i| &\le \|A\|_{L^2\to L^2}\|v_i\|^2= \|A\|_{L^2\to L^2} \sum_{j,k=1}^n q_{ij} q_{ik} \la u_j, u_k\ral\\
&= \|A\|_{L^2\to L^2}\sum_{j=1}^n q_{ij}^2 = \|A\|_{L^2\to L^2}\,.\label{hs4}
\end{align}
Finally, observe  that since $r_i^2$ is the square norm of a Gaussian random varianble in $\mathbb{C}^n$, $\mathbb{E}r_i^2=n$ and 
\begin{equation}\label{hs5}
\ee(A_i') = n\ee(A_i)\,.
\end{equation}
Combining \eqref{hs1}, \eqref{hs2}, \eqref{hs3}, \eqref{hs4} and \eqref{hs5} we get 
\begin{align*}
\pp(|A_i-B|\ge t) &\le \pp(|nA_i - A_i'|\ge nt/2) + \pp(|A_i' - \ee(A_i')|\ge nt/2)\\
&\le \pp(|(r_i^2-n)A_i|\ge nt/2) + \pp(|A_i' - \ee(A_i')|\ge nt/2)\\
&\le \pp(|r_i^2-n|\ge nt/(2\|A\|_{L^2\to L^2})) + \pp(|A_i' - \ee(A_i')|\ge nt/2)\\
&\le C_1 \exp\bigl(-C_2(\|A\|)  \min\{t^2, \, t\} n\bigr)\,,
\end{align*} 
which concludes the proof of the theorem.
\end{proof}

\section{Construction of the perturbed Laplacian}\label{last}

As described in Section \ref{sec:outline}, our strategy will be to break up the spectrum of $-\Delta$ into blocks and to `mix' the eigenfunctions in each block to produce a new operator that is QUE. To do this, it is convenient to work on the spectral side. We will need a few linear algebra lemmas.

Let $\Psi = (\psi_i)_{i\ge 1}$ be a complete orthonormal basis of $L^2(\dbar)$. Let $\Lambda = (\mu_i)_{i\ge 1}$ be a sequence of real numbers. For $s\geq 0$, let $\mf^s(\Psi, \Lambda)$ be the Hilbert space (with complex scalars) consisting of all $f\in L^2(\dbar)$ such that the norm
\[
\|f\|^2_{\mf^s(\Psi,\Lambda)}:=\sum_{i=1}^\infty \la\mu_i\ral^{2s} |\la f,\psi_i\ral|^2 <\infty\,.
\]
Here $\la \mu\ral:=(1+|\mu|^2)^{1/2}.$ For $s<0$, $\mf^s(\Psi,\Lambda):=(\mf^{-s}(\Psi,\Lambda))^*$ is the completion of $L^2(\dbar)$ with respect to $\|\cdot \|_{\mf^s(\Psi,\Lambda)}.$ 
For any $f\in \mf(\Psi, \Lambda):=\mf^1(\Psi,\Lambda)$, the series
\[
T_{\Psi,\Lambda}f:=\sum_{i=1}^\infty \mu_i \la f,\psi_i\ral\psi_i
\]
converges in $L^2(\dbar)=\mf^0(\Psi,\Lambda)$. When $\Psi$ and $\Lambda$ are  clear from context, we will sometimes write $\mf^s$ instead of $\mf^s(\Psi,\Lambda)$. 

\begin{remark}
In our applications, $T_{\Psi,\Lambda}$ will be the Laplacian and $\mf^s(\Psi,\Lambda)$ will be $H^{2s}$. 
\end{remark}

\begin{lmm}\label{compare1}
Let $T_{\Psi, \Lambda}$ be as above. Let $\Lambda' = (\mu_i')_{i\ge 1}$ be another sequence of real numbers. Let $\ep\in (0,1)$ and $\gamma\geq 0$ be numbers such that for all $i$, 
$$|\mu_i' - \mu_i| \le \ep \la\mu_i\ral^{1-\gamma}.$$ 
Then $\|\cdot\|_{\mf^s(\Psi, \Lambda')}$ is equivalent to $\|\cdot\|_{\mf^s(\Psi, \Lambda)}$, and for all $s\in \re$, $T_{\Psi,\Lambda'}-T_{\Psi,\Lambda}:\mf^s(\Psi,\Lambda)\to \mf^{s-1+\gamma}(\Psi,\Lambda)$ with
$$\|T_{\Psi,\Lambda'} - T_{\Psi,\Lambda} \|_{\mf^s(\Psi,\Lambda)\to\mf^{s-1+\gamma}(\Psi,\Lambda)}  \le \ep.$$ 
\end{lmm}
\begin{proof}
Since $\la\mu_i'\ral\le (1+\ep)\la\mu_i\ral$, we have $\|\cdot\|_{\mf^s(\Psi, \Lambda')}\leq C\|\cdot\|_{ \mf^s(\Psi, \Lambda)}$. On the other hand since $\la\mu_i\ral\le \la \mu_i'\ral/(1-\ep)$, so $\|\cdot\|_{\mf^s(\Psi, \Lambda)} \leq C\|\cdot\|_{\mf^s(\Psi, \Lambda')}$. 

Next, let $f\in \mf^s(\Psi,\Lambda)$ with $s\geq 1$.
Then
$$(T_{\Psi,\Lambda}-T_{\Psi,\Lambda'})f=\sum_i (\mu_i-\mu_i')\la f,\psi_i\ral\psi_i.$$ 
Therefore, 
\begin{align*} \|(T_{\Psi,\Lambda}-T_{\Psi,\Lambda'})f\|^2_{\mf^{s-1+\gamma}(\Psi,\Lambda)}&=\sum_i \la \mu_i\ral^{2s-2+2\gamma}|\mu_i-\mu_i'|^2|\la f,\psi_i\ral|^2\\
&\leq \sum_i\la \mu_i\ral^{2s-2+2\gamma}\e^2\la\mu_i\ral^{2(1-\gamma)}|\la f,\psi_i\ral|^2\\
&\leq \e^2\sum_i\la\mu_i\ral^{2s}|\la f,\psi_i\ral|^2\leq \e^2\|f\|_{\mf^{s}(\Psi,\Lambda)}^2\,.
\end{align*}
The density of $\mf(\Psi,\Lambda)$ in $\mf^{s}(\Psi,\Lambda)$ for $s\leq 1$  implies that the result extends to $s\in \re$. 
This concludes the proof of the lemma. 
\end{proof}
\begin{lmm}\label{compare2}
Let $\Psi$ and $\Lambda$ be as above. Let $L$ be the set of distinct elements of $\Lambda$. For each $\ell \in L$, let $I_\ell$ be the set of all $i$ such that $\mu_i = \ell$. Assume that $|I_\ell|$ is finite for each $\ell$. Let $\Psi' = (\psi_i')_{i\ge 1}$ be another complete orthonormal basis, such that for each $\ell\in L$, the span of $(\psi'_i)_{i\in I_\ell}$ equals the span of $(\psi_i)_{i\in I_{\ell}}$. Then for all $s$, $\mf^s(\Psi', \Lambda)=\mf^s(\Psi, \Lambda)$ and $T_{\Psi',\Lambda} = T_{\Psi, \Lambda}$. 
\end{lmm}
\begin{proof}
The lemma follows from the fact that for any $\varphi\in C^\infty(\re)$, 
$$\sum_i \varphi(\mu_i)\langle f, \psi_i\rangle \psi_i=\sum_{\ell\in L}\Pi_{\ell}f$$
where $\Pi_{\ell}$ denotes the orthognonal projection onto the span of $\{\phi_i\mid \mu_i=\ell\}$ and the fact that this span is clearly invariant under choices of bases. 

\begin{comment}
Take some $\ell\in L$. Let $n = |I_\ell|$. Rename the elements of $(\psi_i)_{i\in I_\ell}$ as $\xi_1,\ldots, \xi_n$ and the elements of $(\psi_i')_{i\in I_\ell}$ as $\xi_1',\ldots, \xi_n'$. By assumption, $n$ is finite. Since the span of $(\xi'_i)_{1\le i\le n}$ equals the span of $(\xi_i)_{1\le i\le n}$, there is a matrix $Q = (q_{ij})_{1\le i,j\le n}$ such that for each $i$,
\[
\xi_i' = \sum_{j=1}^n q_{ij} \xi_j\,.
\]
By orthonormality of $\xi_1,\ldots, \xi_n$ and $\xi_1',\ldots, \xi_n'$, 
\begin{align*}
\sum_{k=1}^n q_{ik} q_{jk} &= \la \xi_i', \xi_j'\ral =
\begin{cases}
1 &\text{ if } i=j\, ,\\
0 &\text{ otherwise.}
\end{cases}
\end{align*}
Therefore $Q$ is an orthogonal matrix. Thus, for any $f\in \mf(\Psi, \Lambda)$,
\begin{equation}
\label{eqn:rotate}
\begin{aligned}
\sum_{i=1}^n  \la f,\xi_i'\ral\xi_i' &= \sum_{i=1}^n  \biggl(\sum_{j=1}^n q_{ij} \la f,\xi_j\ral\biggr) \biggl(\sum_{k=1}^n q_{ik} \xi_k\biggr)\\
&= \sum_{j,k=1}^n \biggl(\sum_{i=1}^n q_{ij} q_{ik}\biggr) \la f,\xi_j\ral \xi_k = \sum_{j=1}^n \la f,\xi_j\ral\xi_j\,.
\end{aligned}
\end{equation}
For any function $g:\re\to \re$, \eqref{eqn:rotate} gives
\begin{align*}
\sum_{i\in I_\ell}g( \mu_i )\la \psi_i', f\ral\psi_i' &= g(\ell)\sum_{i=1}^n \la \xi'_i, f\ral\xi'_i\\
&= g(\ell) \sum_{i=1}^n \la\xi_i,f\ral\xi_i= \sum_{i\in I_\ell} g(\mu_i) \la\psi_i, f\ral\psi_i\,.
\end{align*}
Taking $L^2$ norms of both sides with $g(x)=\la x\ral^s$ and $g(x)=x$ respectively we see that $\mf^s(\Psi', \Lambda)=\mf^s(\Psi,\Lambda)$ and $T_{\Psi', \Lambda} = T_{\Psi, \Lambda}$. 
\end{comment}
\end{proof}

\begin{lmm}\label{compare3}
Suppose that $\Lambda$ has $|\mu_i|>c>0$ with $|\mu_i|\to \infty$. 
Let $\gamma_i := 1/\mu_i$ and $\Gamma := (\gamma_i)_{i\ge 1}$. Then $\mf(\Phi, \Gamma)= L^2(\dbar)$ and the range of $T_{\Phi, \Gamma}$ on $L^2$ is contained in $\mf(\Phi, \Lambda)$.  Moreover, $T_{\Phi, \Lambda}T_{\Phi, \Gamma} = I$. 
\end{lmm}
\begin{proof}
If $f\in L^2(\dbar)$, then clearly $f\in \mf(\Phi,\Gamma)$ since $\gamma_i \ra 0$ as $i\ra \infty$. The remainder of the proof follows from elementary computations together with the definition of $\mf(\Phi,\Lambda).$
\end{proof}

Now let $(u_i,\lambda_i^2)_{i\ge 1}$ be a complete orthonormal basis of eigenfunctions of the Dirichlet Laplacian. The let $\Phi=\{u_i\}$ and $\Lambda =\lambda_i^2 $.
\begin{lmm}
\label{inverse}
Let $T_{\Phi,\Gamma}$ be as in Lemma \ref{compare3}.  Then $\mf(\Phi,\Lambda)=\mfd$ and $T_{\Phi,\Gamma}\Delta f=-f$. 
\end{lmm}
\begin{proof}
Lemma \ref{inverse} is also an easy consequence of the spectral theorem applied to the Dirichlet Laplacian.
\end{proof}

We are now ready to construct the perturbed Laplacian and finish the proof of Theorem \ref{quethm2} (and hence, also of Theorem \ref{quethm}).
\begin{proof}[Proof of Theorem \ref{quethm2}]
Let $\{\lambda_i^2\}_{i\ge 1}$ be the eigenvalues of $-\Delta$
and let $\Lambda=\{\lambda_i^2\}_{i\ge1}$. Recall that we assume $\om$ is AQE at scale $\lambda^{-\gamma}$ for some $0\leq \gamma\leq 2$.

Fix $\e\in(0,1)$. Our strategy will be to split the eigenvalues between $(1+\e)^n$ and $(1+\e)^{n+1}$ into $N_n$ intervals where $N_n\sim \e(1+\e)^{n(1-\frac{\gamma}{2})}$. We will then reassign all of the eigenvalues in each subinterval to the left boundary of that interval, randomly rotate the corresponding eigenfunctions, and reassign eigenvalues so that the spectrum is simple. This will produce an almost surely QUE operator.

Observe that for all $i\geq 1$, either $\lambda_i< 1+\e$ or there exist positive integers $n$, $0 \leq j\leq N_n-1$ where 
\begin{equation}
\label{eqn:N}N_n:=\lceil (1+\e)^{n\gamma}\rceil\end{equation} such that 
$$(1+\e)^n\left(1+\frac{j\e}{N_n}\right)\leq \lambda_i< (1+\e)^{n}\left(1+\frac{(j+1)\e}{N_n}\right).$$

In the first case, let $\lambda_i'=\lambda_i$. In the second, let 
$$\lambda_i'=(1+\e)^n\left(1+\frac{j\e}{N_n}\right).$$
Note that for $\e>0$ small enough (independent of $\lambda_i$, 
$$|\lambda_i^2-(\lambda_i')^2|\leq \e(1+\e)^{-n\gamma}(1+\e)^n\leq 3 \e|\lambda_i^2|^{1-\frac{\gamma}{2}}.$$ 
Therefore, by Lemma \ref{compare1}, for $s\geq 0$
$$\mf^s(\Phi,\Lambda')=\mf^s(\Phi,\Lambda)$$
and for $s\geq 1-\frac{\gamma}{2}$ and $\ep$ small enough,
\begin{equation}
\label{e:12}\|T_{\Phi,\Lambda'}-T_{\Phi,\Lambda}\|_{\mf^s\to \mf^{s-1+\frac{\gamma}{2}}}\leq 3\e.
\end{equation}
Let $L$ be the set of distinct eigenvalues in $\Lambda'$. For each $l\in L$, let $I_l$ be the set of $i$ such that $\lambda_i'=l$. Then since $\lambda_i\to\infty$, $|I_l|<\infty$ for all $l$. For each $l$, let $(u_i')_{i\in I_l}$ be a random rotation of $(u_i)_{i\in I_l}.$ Then, by Lemma \ref{compare2}, 
$$T_{\Phi,\Lambda'}=T_{\Phi',\Lambda'},\quad\quad \mf^s(\Phi,\Lambda')=\mf^s(\Phi',\Lambda').$$
Now, for each $ l\in L$, 
\[
l=(1+\e)^n\left(1+\frac{j\e}{N_n}\right)
\]
for some $n,j$ or $0<l<(1+\e)$. Denote the set of $l$ with $0<l<1+\e$ by $L_<$ and let $I_{<} := \cup_{l\in L_<}I_l$. Let
$(\lambda_i'')_{i\in I_<}$ be an arbitrary set of distinct real numbers with 
$$(1-\e)\lambda_i'\leq \lambda_i''<\lambda_i'.$$ 
For $l\notin L_<$, let $(\lambda_i'')_{i\in I_l}$ be an arbitrary set of distinct real numbers with 
$$(1+\e)^n\left(1+\frac{j\e}{N_n}\right)\leq \lambda_i''<(1+\e)^n\left(1+\frac{(j+1)\e}{N_n}\right).$$ 
Then for any $i$, and $\e>0$ small enough (independently of $i$)
$$|(\lambda_i')^2-(\lambda_i'')^2|\leq 3\e |(\lambda_i')^2|^{1-\frac{\gamma}{2}}$$ 
and hence 
$$\mf^s(\Phi',\Lambda'')=\mf^s(\Phi',\Lambda')=\mf^s(\Phi,\Lambda')=\mf^s(\Phi,\Lambda)$$
and
$$\|T_{\Phi',\Lambda''}-T_{\Phi,\Lambda'}\|_{\mf^s\to\mf^{s-1+\frac{\gamma}{2}}}\leq 3\e.$$  
Combining this with \eqref{e:12} gives 
$$\|T_{\Phi',\Lambda''}-T_{\Phi,\Lambda}\|_{\mf^s\to \mf^{s-1+\frac{\gamma}{2}}}\leq 6\e.$$ 
Now, let 
$$\Gamma=
\{\lambda_i^{-1}\}_{i\ge 1}.$$
and $G:=T_{\Phi,\Gamma}.$ For convenience, write $T=T_{\Phi,\Lambda}$ and $T''=T_{\Phi',\Lambda''}.$ 

Then by Lemma \ref{compare3}, $G$ is bounded on $L^2(\dbar)$, has range in $\mf(\Phi,\Lambda)$, and satisfies $TG=I$. Therefore, the operator 
$$S:=(T''-T)G$$ maps $L^2$ into $\mf^{\frac{\gamma}{2}}(\Phi,\Lambda)$.
We will show that $S$ satisfies the three assertions of the theorem. Note that the construction of $S$ involves random rotations and what we will actually show is that $S$ satisfies the conditions with probability one. This will suffice to demonstrate the existence of an $S$ that satisfies the requirements.

First, notice that 
$$\|Sf\|_{\mf^{\frac{\gamma}{2}}}=\|(T''-T)Gf\|_{\mf^{\frac{\gamma}{2}}}\leq 10 \e\|Gf\|_{\mf}\leq C\e\|f\|.$$
Now, by Lemma \ref{inverse}, $\mfd=\mf(\Phi,\Lambda)$, therefore 
\[
\mf^{\frac{\gamma}{2}}(\Phi,\Lambda)=\mfd^{\gamma}=(L^2(\om ),\mfd)_{\frac{\gamma}{2}},
\]
the complex interpolation space of $L^2$ and $\mfd$. Hence (i) holds.

Next, note that by Lemma \ref{inverse} for $f\in \mfd$, $-G\Delta f=f. $ Therefore, for $f\in \mfd$,
$$(I+S)\Delta f=(I+(T''-T)G)\Delta f=T''G\Delta f=-T''f.$$
That is, $-(I+S)\Delta=T''$ on $\mfd$. This proves part (ii) of the theorem. Part (iii) of the theorem follows from the fact that $\{u_i'\}$ is an orthonormal basis for $L^2(\dbar)$ and each $u_i'$ is a linear combination of finitely many $u_i$ which have $u_i\in \mfd^s$ for all $s$. 

It remains to show that the eigenfunctions of $T''$ are equidistributed. For this, recall that 
\[
l=(1+\e)^n\left(1+\frac{j\e}{N_n}\right)
\]
for $l$ large enough and hence 
\begin{gather*} 
I_l=\left\{i:\lambda_-\leq \lambda_i<\lambda_+\right\},\\
\lambda_-:=(1+\e)^n\left(1+\frac{j\e}{N_n}\right),\qquad \lambda_+:=(1+\e)^n\left(1+\frac{(j+1)\e}{N_n}\right).
\end{gather*}

Now,
$$r_+:=\frac{\lambda_+}{\lambda_-}=\frac{1+\frac{(j+1)\e}{N_n}}{1+\frac{j\e}{N_n}}=1+\frac{\e}{N_n}+\O{}(\e^2N_n^{-1}).$$
Then since $\om $ is AQE at scale $\alpha(\lambda)=\O{}(\lambda^{-\gamma})$ and $N_n^{-1}\geq c\lambda^{-\gamma}$,  
\begin{equation}\label{eqn:locWeyl}\lim_{l\in L,\,l\to \infty}\frac{1}{|I_l|}\left|\sum_{i\in I_l}\la (A -\overline{\sigma(A)})1_{\dbar}u_i,1_{\dbar}u_i\ral\right| =0\end{equation}
for $A\in\mc{A}\subset\Psi(\re^d)$, where
\begin{align*}
\overline{\sigma(A)}&=\frac{1}{\vol(1\leq |\xi|\leq 1+r_+)}\iint_{1\leq |\xi|\leq 1+r_+}\sigma(A)(x,\xi)1_{\dbar}dxd\xi\\
&=\int_{S^*\re^d}\sigma(A)(x,\xi)1_{\dbar}dxdS(\xi).
\end{align*}
Note that we have used that $\sigma(A)$ is homogeneous of degree $0$. 

Now, by Theorem \ref{rotthm}, for any $A\in \mc{A}$ and $t\in (0,1)$, 
\begin{align*}
&\mathbb{P}\left(\left|\la A1_{\dbar}u_i',1_{\dbar}u_i'\ral-\frac{1}{|I_l|}\sum_{i\in I_l}\la A1_{\dbar}u_i,1_{\dbar}u_i\ral\right|\geq t\right)\\
&\leq C_1 \exp(-C_2(\|A\|)\min(t^2,t)|I_l|).
\end{align*}
\begin{remark}
Note that we may assume that $u_i$ are real valued without loss of generality.
\end{remark}
So, 
\begin{align*}
&\mathbb{P}\left(\max_{i\in I_l}\left|\la A1_{\dbar}u_i',1_{\dbar}u_i'\ral-\frac{1}{|I_l|}\sum_{i\in I_l}\la A1_{\dbar}u_i,1_{\dbar}u_i\ral\right|\geq t\right)\\
&\leq C_1 |I_l|\exp(-C_2(\|A\|)\min(t^2,t)|I_l|).
\end{align*}
The Weyl law (or more precisely the fact that $\om$ is AQE at scale $\lambda^{-\gamma}$) implies that $|I_l|\sim C\e l^{\frac{d-\gamma}{2}}$ and hence that
$$\sum_{l\in L}|I_l|\exp(-C_2(\|A\|)\min(t^2,t)|I_l|)<\infty.$$ 
Using the Borel--Cantelli lemma we have that 
\begin{align*}
\mathbb{P}&\biggl(\biggl|\la A1_{\dbar}u_i',1_{\dbar}u_i'\ral-\frac{1}{|I_l|}\sum_{i\in I_l}\la A1_{\dbar}u_i,1_{\dbar}u_i\ral\biggr|\geq t\\
&\qquad \qquad \text{ for infinitely many $i$ and $l$ with $i\in I_l$}\biggr)=0.
\end{align*}
Thus, by \eqref{eqn:locWeyl} for all $\delta>0$,
$$\mathbb{P}\left(\limsup_{i\to \infty}\left|\la A1_{\dbar}u_i',1_{\dbar}u_i'\ral-\overline{\sigma(A)}\right|\geq \delta\right)=0.$$
The fact that $\mc{A}$ is dense in $C_0(S^*\om )$ and $C_0(S^*\om)$ is separable then implies that $\mc{M}(u_i')=\{1_{\dbar}dxd\sigma(\xi)\}.$

Now, suppose that $f\in \mfd$ is an $L^2$ normalized eigenfunction of $T''$. Then 
$$0=\|T''f-\lambda^2 f\|^2=\sum_i((\lambda_i'')^2-\lambda^2)^2|\la f,u_i'\ral|^2.$$
Hence, since $f\neq 0$, $\lambda=\lambda_i''$ for some $i$. Thus, for any $j$
\begin{align*}
\la\phi_j',f\ral&=\frac{1}{(\lambda_i'')^2}\la u_j',T''f\ral\\
&=\frac{1}{(\lambda_i'')^2}\sum_{k}\la u_j',u_k'\ral(\lambda_k'')^2\la u_k',f\ral=\frac{(\lambda_j'')^2}{(\lambda_i'')^2}\la u_j',f\ral.
\end{align*}
Hence $\langle u_j',f\rangle=0$ or $\lambda_j''=\lambda_i''$. But for $j$ large enough, $\lambda_i''\neq \lambda_j''$ for $i\neq j$ and hence $f=u_i'$ and $T''$ has equidistributed eigenfunctions.

Notice also that this implies that for $\{f_n\}_{n=1}^\infty$ the eigenfunctions of $-(I+S)\Delta $ with $-(I+S)\Delta f_n=\alpha_n^2f_n$, and $n$ large enough, $f_n=\phi_{n_j}'$ and hence
$$-(I+S)\Delta f_n=T''f_n=\alpha_n^2f_n.$$
Consequently,
\begin{equation}
\label{perturbEst}\|Sf_n\|=\|(T''-T)G\Delta \phi_{n_j}'\|=\|(T''-T)\phi_{n_j}'\|\leq C\e\la \alpha_n\ral^{-\gamma}.\end{equation}
This completes the proof of both Theorem \ref{quethm} and part (ii) of Corollary \ref{almostcor}.
\end{proof}

\begin{proof}[Proof of Corollary \ref{almostcor}]
By Theorem \ref{quethm}, (taking for example, $\e=n^{-1}$) there exists a sequence of linear operators $\{S_n\}_{n\ge 1}$ such that 
\[
\|S_n\|_{L^2\to \mfd^{\gamma}}\ra0
\]
and $-(I+S_n)\Delta$ is positive and has QUE eigenfunctions for each $n$. This implies the existence of an orthonormal basis of $L^2(\dbar)$, $\{f_{n,k}\}_{k=1}^\infty$ and $\alpha_{n,k}$ such that $\|f_{n,k}\|=1$ for each $n$ and $k$, $\alpha^2_{n,k} \ra \infty$ as $k\to \infty$, and
\[
(I+S_n)\Delta f_{n,k} = -\alpha^2_{n,k} f_n\,.
\]

Without loss of generality, $\|S_n\|< 1$. Then the series
\[(I+S_n)^{-1}= \sum_{k=0}^\infty (-1)^k S_n^k
\]
converges in the space of bounded linear operators on $L^2(\dbar)$. Moreover, 
\[
(I+S_n)^{-1}-I=-(I+S_n)^{-1}S_n
\]
Therefore, by \eqref{perturbEst}
\begin{align*}
\|-\Delta f_{n,k} - \alpha^2_{n,k} f_{n,k}\|&= \|\alpha^2_{n,k}(I+S_n)^{-1}S_nf_{n,k}\|\\
&\le \alpha^2_{n,k}\|(I+S_n)^{-1}\|_{L^2\to L^2}\|S_nf_n\|\\\
&\le C\frac{\la \alpha_{n,k}\ral ^{2-\gamma}}{1-\|S_n\|_{L^2\to L^2}}\|S_n\|_{L^2\to \mfd^\gamma}\,.
\end{align*}
Dividing both sides by $\alpha^2_{n,k}$ completes the proof since $\|S_n\|\to 0$.
\end{proof}

\section{Improvements on closed manifolds}
\label{last2}
In order to prove Theorem \ref{quethm} on a manifold $M$ with $\vol(M)=1$, we work with $L^2_0(M)$, the set of $0$ mean functions in $L^2$ to remove the 0 eigenvalue of the Laplacian. Let $\{(u_i, \lambda_i^2)\}_{i=1}^\infty$ be the eigenvalues and eigenfunctions of $-\Delta_g$. Then with $T_{\Phi,\Lambda}$ and $T_{\Phi,\Gamma}$ as above, the proof of Theorem~\ref{quethm} for $M$ proceeds as above.

We now prove Corollary \ref{almostcor2}. For this, we need to use the full strength of Theorem \ref{thm:weylStrong}. 

\begin{proof}[Proof of Corollary \ref{almostcor2}]
Recall that the set of closed geodesics is assumed to have measure zero in $S^*M$. Let $\gamma=1$ and return to \eqref{eqn:N}, where 
 we replace $N_n$ with
$$N_n:=\lceil (1+\e)^{n}\rceil\beta_n$$
where $\beta_n\in \mathbb{N}$ has $\beta_n\to\infty$ slowly enough. We then proceed as in the proof of Theorem \ref{quethm} until \eqref{eqn:locWeyl}. At this point we need to show that there exists $\beta_n\to \infty$ slowly enough so that for $\|f\|_{L^\infty(M)}\leq 1$,
$$\lim_{l\in L,i\to \infty}\frac{1}{|I_l|}\Big|\sum_{i\in I_l}\la (f-\overline{f}) u_i,u_i\ral\Big|=0$$
where 
$$\overline{f}=\int_Mfd\vol.$$ 
First, observe that 
\begin{gather*} \lambda_-:=(1+\e)^{n}\left(1+\frac{j\e}{N_n}\right),\quad\quad\lambda_+:=(1+\e)^{n}\left(1+\frac{(j+1)\e}{N_n}\right),\\
I_l=\left\{i\left|\,\lambda_-\leq \lambda_i<\lambda_+\right.\right\}.\end{gather*}
Note also that by Theorem \ref{thm:weylStrong},
$$\sum_{\lambda_1\leq \lambda_j\leq \lambda_2}|u_j(x)|^2=\frac{(\lambda_2-\lambda_1)\lambda_2^{d-1}}{(2\pi )^d}\vol(S^{d-1})+g(\lambda_2,\lambda_1,x)$$
where
$$\lim_{\lambda_2\to \infty}\sup_{\lambda_1\leq \lambda_2}\|g(\lambda_2,\lambda_1,x)\|_{L^\infty_{x}}\lambda_2^{-d+1}=0.$$
Therefore, integrating, we have
$$\#\{\lambda_1\leq \lambda_j\leq \lambda_2\}=\frac{(\lambda_2-\lambda_1)\lambda_2^{d-1}}{(2\pi )^d}\vol(S^{d-1})+\int g(\lambda_2,\lambda_1,x)dx$$
and, provided that 
$$(\lambda_2-\lambda_1)\lambda_2^{d-1}\gg \sup_{\lambda_1\leq \lambda_2}\|g(\lambda_2,\lambda_1,x)\|_{L^\infty},$$
$$\left|\frac{\sum_{\lambda_1\leq \lambda_j\leq \lambda_2}|u_j(x)|^2}{\#\{\lambda_1\leq \lambda_j\leq \lambda_2\}}-1\right|\leq C\|g(\lambda_2,\lambda_1,x)\|_{L^\infty_{x}}\lambda_2^{-d+1}(\lambda_2-\lambda_1)^{-1} $$
Thus, taking $\lambda_1=\lambda_-$ and $\lambda_2=\lambda_+$, we have 
$$\lambda_2\geq  (1+\e)^{n},\quad\quad \lambda_2-\lambda_1\sim \frac{\e}{\beta_n}.$$
Therefore, taking $\beta_n\to \infty$ slowly enough so that 
$$\lim_{n\to \infty}\sup_{\lambda_-\leq \lambda_+}\|g(\lambda_+,\lambda_-,x)\|_{L^\infty_{x}}\lambda_+^{-d+1}(\lambda_+-\lambda_-)^{-1}=0$$
gives that uniformly for $\|f\|_{L^\infty}\leq 1$,
$$\lim_{l\in L,i\to \infty}\frac{1}{|I_l|}\Big|\sum_{i\in I_l}\la (f-\overline{f}) u_i,u_i\ral\Big|=0.$$
\begin{remark}
Note that the uniformity in $f$ is crucial here and is precisely the reason that we have been unable to prove a version of Corollary \ref{almostcor2} giving an orthonormal basis of QUE eigenfunctions. More precisely, the remainder in the Weyl law involving matrix elements $\langle Au_j,u_j\rangle$ depends on more than just $\sup |\sigma(A)|$. In particular, it involves derivatives $\sigma(A)$. 
\end{remark}

Then, using the fact that $f\in C^\infty(M)$ with $\|f\|_{L^\infty(M)}\leq 1$ is dense in the unit ball of the dual space to finite radon measures, that this space is separable, and following the proof of Theorem \ref{quethm} from \eqref{eqn:locWeyl} shows that for all $\e>0$, there exists $S:L^2(M)\to H^1(M)$ so that $\|S\|_{L^2\to H^1}\leq \e$, 
$-(I+S)\Delta_g$ has equidistributed eigenfunctions, $\{(f_n,\alpha_n)\}_{n=1}^\infty,$ and by \eqref{perturbEst} $\|S f_n\|=\o{}(\alpha_n^{-1})\|f_n\|.$

Therefore, 
$$-(I+S)\Delta f_n=\alpha_n^2f_n.$$
Now, 
$$(I+S)^{-1}=\sum_{k=0}^\infty (-1)^kS^k,\quad\quad (I+S)^{-1}-I=-(I+S)^{-1}S.$$ 
Therefore, 
$$(-\Delta -\alpha_n)f_n=-\alpha_n^2(I+S)^{-1}Sf_n$$
and hence, 
\begin{align*} 
\|(-\Delta-\alpha_n^2)f_n\|&\leq |\alpha_n^2|\|(I+S)^{-1}\o{}(\alpha_n^{-1})\|f_n\|\\
&=\o{}(\alpha_n)\|f_n\|
\end{align*}
Dividing by $\alpha_n^2$ completes the proof of the corollary.
\end{proof}

\vskip.2in
\noindent {\bf Acknowledgements.} The authors would like to thank Persi Diaconis, Peter Sarnak, Andr\'as Vasy, Steve Zelditch and Maciej Zworski  for various helpful discussions as well as Dima Jakobson for his careful reading of a previous version. The authors also thank the anonymous referee for many useful comments suggestions.

\appendix 

\section{A local Weyl law on regular domains}\label{ptwisesec}

Throughout this section, we assume that $\Omega\subset \re^d$ is a regular domain. Let $B_t$ be a standard $d$-dimensional Brownian motion (in $\mathbb{R}^d$), starting at some point $x\in \om$. Recall the definition \eqref{exit} of the exit time $\tau_\om$ from the domain $\om$.  %Let $X_t$ be a new process, that is equal to $B_t$ if $t<\tau_\om$, and given a special value $\dagger$ (usually called the ``cemetery state'') if $t\ge \tau_\om$. This process is called Brownian motion killed upon exiting $\om$, or simply, ``killed Brownian motion''. 
We will need a few well-known facts about this exit time, summarized in the following theorem.
\begin{thm}[Compiled from Proposition 4.7 and Theorems 4.12 and 4.13 of Chapter II in \citet{bass95} and Section 4 of Chapter 2 in \citet{portstone78}]\label{summarythm}
For any regular domain $\om$ (as defined in Section \ref{intro}), there exists a unique function $p:(0,\infty)\times \dbar\times \dbar \ra[0,\infty)$ such that: 
\begin{enumerate}
\item[\textup{(i)}] For any bounded Borel measurable $f:\om\ra\rr$, $x\in \om$, and $t\geq 0$, 
\[
\ee^x(f(B_t); t< \tau_\om) = \int_\om p(t,x,y) f(y)\, dy\,,
\]
where $\ee^x$ denotes expectation with respect to the law of Brownian motion started at $x$.
\item[\textup{(ii)}] There is a complete orthonormal basis $(u_i)_{i\ge 1}$ of $L^2(\dbar)$ such that each $u_i$ is $C^\infty$ in $\om$, vanishes continuously at the boundary, and there are numbers $0<\lambda_1^2\le \lambda_2^2\le\cdots$ tending to infinity such that 
\[
p(t,x,y) = \sum_{i=1}^\infty e^{-\frac{1}{2}\lambda_i^2 t} u_i(x)u_i(y)\,,
\]
where the right side converges absolutely and uniformly on $\dbar\times \dbar$. Moreover, $-\Delta_D u_i = \lambda_i^2 u_i$ for each $i$ where $-\Delta_D$ is the Dirichlet Laplacian on $\Omega$.
\end{enumerate}
\end{thm} 

\begin{remark}
Notice that $p(t,x,y)$ is the Heat kernel of $\Omega$. That is, the kernel of $e^{t\Delta_D/2}.$ 
\end{remark}
%Note that by the assumption that $\vol(\partial \om )=0$, it follows that $\tau_\om$ is a continuous random variable. In particular, it does not matter whether we put $t<\tau_\om$ or $t\le \tau_\om$ in part (i) of the above theorem. 
Let $\lambda_i$ be as in the above theorem. For each $\ep>0$ and $\lambda >0$ define a set of indices $J_{\ep,\lambda}$ as 
\[
J_{\ep,\lambda} := \{i: \lambda \le \lambda_i< \lambda (1+\ep)\}\,.
\]
Let $|J_{\ep,\lambda}|$ denote the size of the set $J_{\ep,\lambda}$. The following theorem is the main result of this section. 
\begin{thm}\label{ptwisethm}
For any fixed $\ep>0$, $J_{\ep, \lambda}$ is nonempty for all large enough $\lambda$ and for $A\in \Psi(\re^d)$, with symbol $\sigma(A)(x,\xi)$ supported in $K_x\times \re^d$ with $K_x\subset \om$ compact, 
\begin{align*}
\lim_{\lambda \ra \infty} \biggl|\frac{1}{|J_{\ep,\lambda}|}\sum_{i\in J_{\ep,\lambda}} \la (A-\bar{A})1_{\dbar}\phi_ j,1_{\dbar}\phi_j\ral \ \biggr|\,  =0\,.
\end{align*}
where 
$$\bar{A}=\int_{S^*\re^d}\sigma(A)(x,\xi)1_{\dbar}dxdS(\xi)$$
where $S$ is the normalized surface measure on $S^{d-1}$. 
Moreover, $$\lim_{\lambda\ra\infty} \lambda^{-d}|J_{\ep, \lambda}| = \frac{(1+\e)^{d}-1}{(4\pi)^{d/2}\Gamma(d/2+1)}.$$ 
\end{thm}
This theorem implies Theorem \ref{localWeyl} since $\sigma(A)$ is homogeneous of degree $0$ and is a variant of results that are sometimes called `local Weyl laws', as in  \cite{zelditch10}. However, we are not aware of a local Weyl law in the literature that applies for a domain as general as the one considered here. Our proof follows closely that in \cite{gerardleichtnam93}, but by using probabilistic methods to obtain estimates on the kernel of $e^{t\Delta_D}$, we are able to weaken the regularity assumptions on the domain.

Since we will have occasion to refer to both the Laplace operator on $L^2(\re^d)$ and the Dirichlet Laplacian in this section, we will denote them respectively by $-\Delta_{\re^d}$ and $-\Delta_D$. Theorem \ref{ptwisethm} will follow from the following lemma 
\begin{lmm}
\label{traceLem}
Take $A \in \Psi^0(\re^d)$ with symbol $\sigma(A)(x,\xi)$ supported in $K_x\times \re^d$ where $K_x\subset \om$ is compact. Then for all $t>0$, $1_{\dbar}A1_{\dbar}e^{t\Delta_{D}}$ is trace class as an operator on $L^2(\dbar)$ and 
$$\lim_{t\to 0^+}\frac{\tr(1_{\dbar}A1_{\dbar}e^{t\Delta_D})}{\tr (e^{t\Delta_D})}=\int_{S^*\re^d}\sigma(A)(x,\xi)1_{\dbar}dxdS(\xi)$$
where $\lambda=1_{\dbar}dxd\sigma(\xi)$ and $\sigma$ is the normalized surface measure on $S^{d-1}$. 
\end{lmm}
We first show how Theorem \ref{ptwisethm} follows from Lemma \ref{traceLem}. We will need the following classical Tauberian theorem (see for example \cite{Ta}).
\begin{lmm}
\label{taubLem}
Suppose that $F:[0,\infty)\to \re$ is nondecreasing and for some $A,\gamma>0$,
$$\int_0^\infty e^{-t\alpha}dF(\alpha)\sim At^{-\gamma} \ \ \text{ as } t\to 0^+\,.$$
Then
$$F(\tau)\sim \frac{A\tau^{\gamma}}{\Gamma(\gamma+1)}\ \  \text{ as } \tau \to \infty\,.$$ 
\end{lmm}

The rest of this section is devoted to the proof of Lemma \ref{traceLem} and Theorem \ref{ptwisethm}. We will freely use the notation introduced in the statements of Theorem \ref{summarythm} and Theorem \ref{ptwisethm} without explicit reference. 
First, note that the following corollary of Theorem \ref{summarythm} is immediate from the continuity of $p$.
\begin{lmm}\label{immcor}
Take any $x,y\in \om$ and let $A_{y,r}$ be the closed ball of radius $r$ centered at $x$. Then
\[
p(t,x,y) = \lim_{r\ra0} \frac{\pp^x(B_t\in A_{y,r},\, t<\tau_\om)}{\vol(A_{y,r})}\,.
\]
\end{lmm}
\begin{proof}
By assertion (i) of Theorem \ref{summarythm}, 
\[
\pp^x(B_t\in A_{y,r},\, t<\tau_\om) = \int_{A_{y,r}} p(t,x,z)\, dz\,.
\]
By assertion (ii) of Theorem \ref{summarythm}, 
\[
\lim_{r\ra0}\frac{1}{\vol(A_{y,r})} \int_{A_{y,r}} p(t,x,z)\, dz = p(t,x,y)\,.
\]
The proof is completed by combining the two displays.
\end{proof}
The following lemma compares the transition density of killed Brownian motion with the transition density of unrestricted Brownian motion when $t$ is small.
\begin{lmm}\label{pdifflmm}
Let
\[
\rho(t,x,y) := \frac{1}{(2\pi t)^{d/2}} e^{-\|x-y\|^2/2t}
\]
be the transition density of Brownian motion. Take any $x,y\in \om$ and let $\delta_y$, $\delta_x$ be respectively the distance of $y$ and $x$ from $\partial \om$. 
Then 
\begin{gather*}
|\partial_y^\alpha(\rho(t,x,y) - p(t,x,y)) | \le \frac{Ce^{-\delta_y^2/2t}}{t^{d/2+|\alpha|}}\,,\quad \quad  0< t<\delta_y^2/(d+2|\alpha|), \\
 |\partial_x^\alpha(\rho(t,x,y) - p(t,x,y)) | \le \frac{Ce^{-\delta_x^2/2t}}{t^{d/2+|\alpha|}}\,,\quad\quad  0< t<\delta_x^2/(d+2|\alpha|), 
\end{gather*}
where $C$ is  a finite constant that depends only on $d$, $|\alpha|$ and the diameter of the domain $\om$.
\end{lmm}
\begin{proof}
Since $\tau_\om$ is a stopping time, the strong Markov property of Brownian motion implies that $X_s := B_{s+\tau_\om}$ is a standard Brownian motion started from $B_{\tau_\om}$ that is independent of the stopped sigma algebra of $\tau_\om$, which we will  denote by $\mf_{\tau_\om}$. Consequently, if $A_{y,r}$ is the closed ball of radius $r<\delta_y/2$ centered at $y$, then for any $s\ge 0$, 
\begin{align*}
\pp^x(X_{s} \in A_{y,r}\mid\mf_{\tau_\om}) &=  \frac{1}{(2\pi s)^{d/2}}\int_{A(y,r)} e^{-\|z-B_{\tau_\om}\|^2/2s}\, dz\,.
\end{align*}
Consequently, %if $0< t\le (\delta_x-r)^2/d$,  then
\begin{align*}
&\pp^x(B_t \in A_{y,r}, \, t\ge \tau_\om) = \pp^x(X_{t-\tau_\om} \in A_{y,r},\, t\ge \tau_\om)\\
&= \ee^x(\pp^x(X_{t-\tau_\om} \in A_{y,r}\mid\mf_{\tau_\om})\,;\, t\ge \tau_\om)\\
&= \ee^x\biggl(\frac{1}{(2\pi (t-\tau_\om))^{d/2}}\int_{A(y,r)} e^{-\|z-B_{\tau_\om}\|^2/2(t-\tau_\om)}\, dz\,;\, t\ge \tau_\om\biggr)\,,
\end{align*}
where the term inside the expectation is interpreted as zero if $t=\tau_\om$. 
Dividing both sides by $\vol(A_{y,r})$, sending $r$ to zero, and observing that the term inside the above expectation after division by $\vol(A_{y,r})$ is uniformly bounded by a deterministic constant, we get
\begin{align*}
&\lim_{r\ra 0} \frac{\pp^x(B_t \in A_{y,r}, \, t\ge \tau_\om)}{\vol(A_{y,r})}\\
&= \ee^x\biggl(\frac{1}{(2\pi (t-\tau_\om))^{d/2}} e^{-\|y-B_{\tau_\om}\|^2/2(t-\tau_\om)}\,;\, t\ge \tau_\om\biggr)\,.
\end{align*}
Now note that
\begin{align*}
\pp^x(B_t \in A_{y,r}) - \pp^x(B_t\in A_{y,r}, \, t <\tau_\om)&= \pp^x(B_t\in A_{y,r}, \, t \ge \tau_\om)
\end{align*}
and 
\begin{align*}
\lim_{r\ra0} \frac{\pp^x(B_t\in A_{y,r})}{\vol(A_{y,r})} =\rho(t,x,y)\,,
\end{align*}
and by Lemma \ref{immcor},
\begin{align*}
\lim_{r\ra0} \frac{\pp^x(B_t\in A_{y,r},\, t <\tau_\om)}{\vol(A_{y,r})} =p(t,x,y)\,.
\end{align*}
Combining all of the above observations, we get
\begin{align*}
\rho(t,x,y) - p(t,x,y) &= \ee^x\biggl(\frac{1}{(2\pi (t-\tau_\om))^{d/2}} e^{-\|y-B_{\tau_\om}\|^2/2(t-\tau_\om)}\,;\, t\ge \tau_\om\biggr)\,.
\end{align*}
Now note that any derivative of the term inside the expectation (with respect to $y$) is uniformly bounded by a deterministic constant that does not depend on $y$ or $t$. Therefore derivatives with respect to $y$ can be carried inside the expectation. Consequently,
\begin{align*}
&\partial_y^{|\alpha|}\rho(t,x,y) - \partial_{y}^{|\alpha|} p(t,x,y) \\
&= \ee^x\biggl(\frac{1}{(2\pi (t-\tau_\om))^{d/2}} \partial_y^{|\alpha|}(e^{-\|y-B_{\tau_\om}\|^2/2(t-\tau_\om)})\,;\, t\ge \tau_\om\biggr)\,.%\\
%&= \ee^x\biggl(\frac{p_T(y)}{(2\pi (t-\tau_\om))^{d/2}} e^{-\|y-B_{\tau_\om}\|^2/2(t-\tau_\om)}\,;\, t\ge \tau_\om\biggr)\,,
\end{align*} 
If $t\le \delta_y^2/(d+2|\alpha|)$, an easy verification shows that
\begin{align*}
&\biggl|\frac{1}{(2\pi (t-\tau_\om))^{d/2}} \partial_y^{|\alpha|}(e^{-\|y-B_{\tau_\om}\|^2/2(t-\tau_\om)})\biggr|\\
&\le \frac{C}{(t-\tau_\om)^{d/2+|\alpha|}} e^{-\|y-B_{\tau_\om}\|^2/2(t-\tau_\om)}\,,
\end{align*}
where $C$ depends only on $d$, $|\alpha|$ and the diameter of the domain $\om$.
Another easy calculation shows  that the map $u\mapsto (2\pi u)^{-d/2-|\alpha|} e^{-\beta^2/2u}$ is increasing in $u$ when $0< u\le \beta^2/(d+2|\alpha|)$. Therefore if $\tau_\om<t\le \delta_y^2/(d+2|\alpha|)$, then
\[
\frac{1}{(t-\tau_\om)^{d/2+|\alpha|}} e^{-\|y-B_{\tau_\om}\|^2/2(t-\tau_\om)}\le \frac{ e^{-\delta_y^2/2t}}{t^{d/2+|\alpha|}}\,.
\]
Noticing that $p(t,x,y)=p(t,y,x)$ (for example, by part (ii) of Theorem \ref{summarythm}) and $\rho(t,x,y)=\rho(t,y,x)$, this completes the proof of the lemma.
\end{proof}

\begin{proof}[Proof of Theorem \ref{ptwisethm} from Lemmas \ref{traceLem} and \ref{pdifflmm}]
By Lemma \ref{traceLem}, we have that 
\begin{equation}\label{eqn:tr}\frac{\tr(1_{\dbar}A1_{\dbar}e^{t\Delta_D})}{\tr (e^{t\Delta_D})}\to\int_{S^*\re^d}\sigma(A)(x,\xi)1_{\dbar}dxdS(\xi),\quad\quad t\to 0^+.\end{equation}
Since $\{u_j\}_{j\ge 1}$ is an orthonormal basis of $L^2(\dbar)$, 
\begin{equation}\label{eqn:trexp}\tr(1_{\dbar}A1_{\dbar}e^{t\Delta_D})=\sum_{j}e^{-t\lambda_j^2}\la A1_{\dbar}u_j,1_{\dbar}u_j\ral.\end{equation}
By Lemma \ref{pdifflmm} and the assumption that $\vol(\dbar)=1$,
\begin{equation}\label{eqn:trnoA}\tr(e^{t\Delta_D})=\sum e^{-t\lambda_j^2}=\int_\om p(2t,x,x)\sim (4\pi t)^{-d/2}\to \infty ,\quad t\to 0^+.\end{equation}
Putting \eqref{eqn:tr}, \eqref{eqn:trexp}, and \eqref{eqn:trnoA} together we have that 
$$\sum_{j}e^{-t\lambda_j^2}\la A1_{\dbar}u_j,1_{\dbar}u_j\ral\sim (4\pi t)^{-d/2}\int_{S^*\re^d}\sigma(A)(x,\xi)1_{\dbar}dxdS(\xi).$$ 
Now, assuming that $\sigma(A)\geq 0$, and adding a regularizing perturbation  $C\in \Psi^{-1}$ if necessary, so that $(A+C)\geq 0$, we may apply Lemma \ref{taubLem} with 
$$F_{A}(\tau)=\sum_j 1_{\lambda_j\leq \tau}\la (A+C)1_{\dbar}u_j,1_{\dbar}u_j\ral.$$
More precisely, we  apply it with 
$$\tilde{F}_{A}(\tau)=\sum_j 1_{\lambda_j\leq \sqrt{\tau}}\la (A+C)1_{\dbar}u_j,1_{\dbar}u_j\ral$$
and rescale so that 
$$F_A(\tau)\sim \frac{\tau^{d}}{(4\pi )^{d/2}\Gamma(d/2+1)}\int_{S^*\re^d}\sigma(A)(x,\xi)1_{\dbar}dxdS(\xi).$$ 
Now, $C:L^2(\re^d)\to L^2(\re^d)$ is compact. Therefore, $\lim_{j\to \infty}\la C1_{\dbar}u_j,1_{\dbar}u_j\ral =0$, and hence 
$$\sum1_{\lambda_j\leq \tau}\la C1_{\dbar}u_j,1_{\dbar}u_j\ral=\o{}(\tau^{d}).$$ 
Therefore
\begin{align}
&\sum_j1_{\lambda_j\leq \tau}\la A1_{\dbar}u_j,1_{\dbar}u_j\ral \nonumber \\
&\sim \frac{ \tau^{d}}{(4\pi)^{d/2}\Gamma(d/2+1)}\int_{S^*\re^d}\sigma(A)(x,\xi)1_{\dbar}dxdS(\xi)+o(\tau^d).\label{eqn1}
\end{align}
Using \eqref{eqn:trnoA} together with Lemma \ref{taubLem} also gives
\begin{equation}\label{eqn2}\#\{\lambda_j:\lambda_j\leq \tau\}\sim \frac{\tau^{d}}{(4\pi)^{d/2}\Gamma(d/2+1)}.\end{equation}
Subtraction of two formulae like \eqref{eqn1} and \eqref{eqn2} yields the desired asymptotics.
\end{proof}
The proof of Lemma \ref{traceLem} requires one further lemma.
\begin{lmm}\label{traceLemLem}
Let $A\in \Psi(\re^d)$ with symbol compactly supported in $\Omega$ and $\psi\in C_c^\infty(\Omega)$ with $\psi\equiv 1$ on $\supp \sigma(A)$. Then we have 
$$(4\pi t)^{d/2}\tr (A\psi e^{t\Delta_{\re^d}}\psi)\to \int_{S^*\re^d}\sigma(A)(x,\xi)1_{\dbar}dxdS(\xi) \ \ \text{ as } t\to 0^+$$
and there exists $\e>0$, $t_0>0$, so that for $0<t<t_0$,
$$|\tr A\psi(e^{t\Delta_D}-e^{t\Delta_{\re^d}})\psi|\leq \e^{-1}e^{-\e/t}\,.$$
\end{lmm}
\begin{proof}
The kernel $K(t,x,y)$ of $A\psi e^{t\Delta_{\re^d}}\psi$ is given by 
\begin{align*} 
K(t,x,y)&=(2\pi )^{-d}\int a(x,\xi)\int e^{i\la x-w,\xi\ral -|w-y|^2/4t}(4\pi t)^{-d/2}\psi(w)\psi(y)dwd\xi\\
&=(2\pi )^{-2d}\int a(x,\xi)\int e^{i\la x,\xi\ral}e^{-|\xi-\eta|^2t}e^{-i\la y,\xi-\eta\ral}\hat{\psi}(\eta)\psi(y)d\eta d\xi.
\end{align*}
where 
\begin{align*}
&\left|\partial_x^\alpha\partial_\xi^\beta\left(a-\sum_{j=0}^{N-1}a_j(x,\xi)\right)\right|\leq C_{\alpha\beta}\langle \xi\rangle^{-N},\\
&\qquad \qquad  a_j\in S^j_{hom}(T^*\re^d),\quad a_0(x,\xi)=\sigma(A)(x,\xi)
\end{align*}
So, changing variables so that $\xi \sqrt{t}=\zeta$,
\begin{align*}
&(4\pi t)^{d/2}\tr \psi e^{it\Delta_{\re^d}}\psi\\
& =\pi^{-d/2}t^{d/2}(2\pi )^{-d}\int a(x,\xi)\int e^{-|\xi-\eta|^2t}e^{i\la x,\eta\ral}\hat{\psi}(\eta)\psi(x)d\eta d\xi dx\\
&=\pi^{-d/2}(2\pi )^{-d}\int a(x,\zeta t^{-1/2})\int e^{-|\zeta-\eta\sqrt{t}|^2}e^{i\la x,\eta\ral}\hat{\psi}(\eta)\psi(x)d\eta d\zeta dx
\end{align*}
Now, since $\psi \in \mc{S}$, we can use the dominated convergence theorem and let $t\to 0^+$ to obtain
\begin{align*}
&\lim_{t\to 0^+}(4\pi t)^{d/2}\tr A\psi e^{t\Delta_{\re^d}}\psi\\
&=\pi^{-d/2}(2\pi )^{-d}\int \sigma(A)(x,\zeta )\int e^{-|\zeta|^2}e^{i\la x,\eta\ral}\hat{\psi}(\eta)\psi(x)d\eta d\zeta dx\\
&=\pi^{-d/2}\int \sigma(A)(x,\zeta )\int e^{-|\zeta|^2}\psi(x)\psi(x)d\zeta dx\\
&=\pi^{-d/2}\int \sigma(A)(x,\zeta )\int e^{-|\zeta|^2}d\zeta dx
\end{align*}
where we have used that $\psi\equiv 1$ on $\supp \sigma(A)$. Now, since $\sigma(A)$ is homogeneous of degree 0 in $\zeta$, this is equal to 
\begin{align*}
&\pi^{-d/2}\vol(S^{d-1})\int_{S^*\re^d} \sigma(A)(x,\zeta )1_{\dbar}dxdS(\zeta) \int_0^\infty e^{-r^2}r^{d-1} dr\\
&=\int_{S^*\re^d} \sigma(A)(x,\zeta )1_{\dbar}dxdS(\zeta).
\end{align*}
For the second claim, we use Lemma \ref{pdifflmm}. Let $g(t,x,y)$ denote the kernel of $\psi (e^{t\Delta_D}-e^{t\Delta})\psi$ and $g_A(t,x,y)$ the kernel of $A\psi (e^{t\Delta_D}-e^{t\Delta})\psi$. Then
$$\tr (A\psi (e^{t\Delta_D}-e^{t\Delta_{\re^d}})\psi)=\int_{\dbar} g_A(t,x,x)dx.$$ 
Using the Sobolev embedding, for $m>\frac{d}{2}$,
$$|g_A(t,x,x)|\leq \|g_A\|_{H^m}\leq \|A\|_{H^m\to H^m}\|g(t,\cdot,y)\|_{H^m}.$$
Letting $\delta_x=d(\supp \psi, \partial \om )$, Lemma \ref{pdifflmm} implies for $t<\frac{\delta_x^2}{2(d+2|\alpha|)}$,
$$|\partial_x^{\alpha} g(t,x,y)|\leq C\frac{e^{-\delta_x^2/4t}}{t^{d/2+|\alpha|}}$$
and hence since $\|A\|_{H^m\to H^m}<\infty$, 
$$|g_A(t,x,x)|\leq C\frac{e^{-\delta_x^2/4t}}{t^{d/2+m}}$$
for each  $t<\frac{\delta_x^2}{2(d+2m)}$. 
\end{proof}
We are now ready to prove Lemma \ref{traceLem}.
\begin{proof}[Proof of Lemma \ref{traceLem}]
Since $1_{\dbar}A1_{\dbar}:L^2(\om )\to L^2(\om )$, and $e^{t\Delta_D}$ is trace class, $1_{\dbar}A1_{\dbar}e^{t\Delta_D}$ is trace class. Let $\psi\in C_c^\infty(\om )$ with $\psi \equiv 1$ on $\supp \sigma(A)$. Then, 
$$\psi A= A-(1-\psi)A,\quad\quad A\psi=\psi A+[A,\psi].$$
But, $(1-\psi)A,$ $[A,\psi]\in \Psi^{-1}$ and hence $1_{\dbar}(1-\psi)A$, $1_{\dbar}[A,\psi]$ are compact on $L^2(\re^d)$ and have 
$$\|1_{\dbar}(1-\psi)A1_{\dbar}\phi_k\|+\|1_{\dbar}[A,\psi]1_{\dbar}\phi_k\|\to 0,\quad\quad k\to \infty.$$ 
In particular, 
$$\frac{\tr(1_{\dbar}A1_{\dbar}e^{t\Delta_D})}{\tr{e^{t\Delta_D}}}\sim\frac{\tr (\psi A\psi e^{t\Delta_D})}{\tr e^{t\Delta_D}}=\frac{\tr (A\psi e^{t\Delta_D}\psi)}{\tr e^{t\Delta_D}},\quad\quad t\to 0^+.$$ 
By Lemma \ref{traceLemLem}, the proof of Lemma \ref{traceLem} is now complete.
\end{proof}

\bibliographystyle{plainnat}

\end{document}